\documentclass[12pt,english,fleqn,liststotoc,bibtotoc,idxtotoc,BCOR7.5mm,tablecaptionabove]{amsart}
\usepackage[T1]{fontenc}
\usepackage[latin1]{inputenc}
\usepackage{geometry}
\usepackage{color}
\usepackage{amsthm}
\usepackage{amstext}
\usepackage{amssymb}
\usepackage{amsmath}
\usepackage{graphicx}
\usepackage{young}
\usepackage{xypic}
\usepackage{amscd}
\usepackage[section]{placeins}
\usepackage[svgnames]{xcolor}
\usepackage[unicode=true,
bookmarks=true,bookmarksnumbered=true,bookmarksopen=false,
breaklinks=false,pdfborder={0 0 1},backref=false,colorlinks=true]
{hyperref}
\hypersetup{pdftitle={},
	pdfauthor={Nick Early},
	pdfsubject={},
	pdfkeywords={},
	linkcolor=black, citecolor=black, urlcolor=blue, filecolor=blue,  pdfpagelayout=OneColumn, pdfnewwindow=true,  pdfstartview=XYZ, plainpages=false, pdfpagelabels}
\usepackage{breakurl}

\usepackage{scalerel}

\makeatletter

\theoremstyle{plain}
\newtheorem{thm}{\protect\theoremname}
\theoremstyle{plain}
\newtheorem{conjecture}[thm]{\protect\conjecturename}
\theoremstyle{plain}

\theoremstyle{remark}

\theoremstyle{plain}

\theoremstyle{plain}
\newtheorem{prop}[thm]{\protect\propositionname}
\theoremstyle{remark}

\theoremstyle{remark}
\newtheorem{rem}[thm]{\protect\remarkname}
\theoremstyle{definition}
\newtheorem{defn}[thm]{\protect\definitionname}
\theoremstyle{definition}
\newtheorem{example}[thm]{\protect\examplename}
\theoremstyle{plain}
\newtheorem{cor}[thm]{\protect\corollaryname}
\theoremstyle{plain}

\numberwithin{thm}{section}

\usepackage{ifpdf}
\ifpdf

\IfFileExists{lmodern.sty}
{\usepackage{lmodern}}{}

\fi 

\usepackage{multicol}

\newenvironment{nouppercase}{%
	\renewcommand{\uppercasenonmath}[1]{}}{}

\makeatother

\providecommand{\claimname}{\inputencoding{latin9}Claim}
\providecommand{\conjecturename}{\inputencoding{latin9}Conjecture}
\providecommand{\corollaryname}{\inputencoding{latin9}Corollary}
\providecommand{\definitionname}{\inputencoding{latin9}Definition}
\providecommand{\examplename}{\inputencoding{latin9}Example}
\providecommand{\lemmaname}{\inputencoding{latin9}Lemma}
\providecommand{\notename}{\inputencoding{latin9}Note}
\providecommand{\propositionname}{\inputencoding{latin9}Proposition}
\providecommand{\questionname}{\inputencoding{latin9}Question}
\providecommand{\remarkname}{\inputencoding{latin9}Remark}
\providecommand{\theoremname}{\inputencoding{latin9}Theorem}
\providecommand{\problemname}{\inputencoding{latin9}Problem}

\newcommand\twoheaduparrow{\mathrel{\rotatebox{90}{$\twoheaduparrow$}}}
\newcommand\twoheaddownarrow{\mathrel{\rotatebox{270}{$\twoheaddownarrow$}}}

\begin{document}
	\author{N\MakeLowercase{ick} E\MakeLowercase{arly}}
	\thanks{Max Planck Institute for Mathematics in the Sciences.  Email: \href{nick.early@mis.mpg.de}{nick.early@mis.mpg.de}}
		\title[Factorization for Generalized Biadjoint Scalar Amplitudes via Matroid Subdivisions]{Factorization for Generalized Biadjoint Scalar Amplitudes via Matroid Subdivisions}
	\begin{nouppercase}
		\maketitle
	\end{nouppercase}
	\begin{abstract}		
		We study the problem of factorization for residues of generalized biadjoint scalar scattering amplitudes $m^{(k)}_n$, introduced by Cachazo, Early, Guevara and Mizera (CEGM), involving multi-dimensional residues which factorize generically into $k$-ary products of lower-point generalized biadjoint amplitudes of the same type $m^{(k)}_{n_1}\cdots m^{(k)}_{n_k}$, where $n_1+\cdots +n_k = n+k(k-1)$, noting that smaller numbers of factors arise as special cases.  Such behavior is governed geometrically by regular matroid subdivisions of hypersimplices and cones in the positive tropical Grassmannian, and combinatorially by collections of compatible decorated ordered set partitions, considered modulo cyclic rotation.  We make a proposal for conditions under which this happens and we develop $k=3,4$ in detail.  We conclude briefly to propose a novel formula to construct coarsest regular matroid subdivisions of all hypersimplices $\Delta_{k,n}$ and rays of the positive tropical Grassmannian, which should be of independent interest.

	\end{abstract}

	\begingroup
	\let\cleardoublepage\relax
	\let\clearpage\relax
	\tableofcontents
	\endgroup

\section{Introduction}

	Recently, building on the work of Cachazo, He and Yuan \cite{CHY2014,CHY2014C}, Cachazo, Early, Guevara and Mizera (CEGM), discovered in \cite{CEGM2019} a recursive family of inter-related generalized biadjoint scattering amplitudes $m^{(k)}_n$.  They fall into a symmetric triangular hierarchy indexed by pairs $(k,n)$ of integers satisfying $2\le k\le n-2$,
	$$\begin{array}{ccccccc}
		&  &  & m^{(2)}_4 &  &  &  \\
		&  & m^{(2)}_5 &  & m^{(3)}_5 &  &  \\
		& m^{(2)}_6 &  & m^{(3)}_6 &  & m^{(4)}_6 &  \\
		m^{(2)}_7 &  & m^{(3)}_7 &  & m^{(4)}_7 &  & m^{(5)}_7 \\
		&  &  & \vdots &  &  & 
	\end{array}$$
	There are numerous upward-looking relations in the hierarchy; and there is the standard (horizontal) duality $m^{(k)}_n(\mathfrak{s}_J) = m^{(n-k)}_{n}(\mathfrak{s}_{J^c})$; but also by taking residues, the entries in the triangle are directly related to the biadjoint scalars $m^{(2)}_n$ on the boundary \cite{CE2022}.  Moreover $m^{(k)}_n$ is seen to act directly as a kind of interpolation between $n$ copies of $m^{(k-1)}_{n-1}$ and $n$ copies of $m^{(k)}_{n-1}$, one for each facet of the hypersimplex $\Delta_{k,n} = \left\{x\in \lbrack 0,1\rbrack^n:\sum_{j=1}^n x_j=k \right\}$.  Physically speaking, these $2n$ ``boundaries'' are realized as certain hard and soft limits \cite{CUZ2019,SG2019}, see also \cite{LikelihoodDegenerations}.  There is a generalized notion of Feynman diagrams \cite{BC2019} for the theory, given by \textit{collections} of metric trees \cite{HJJS}, and in terms of matroid subdivisions and the positive tropical Grassmannian $\text{Trop}^+G(k,n)$ via matroidal arrangements of cyclically skewed, affinely translated tropical hyperplanes which we call blades, in \cite{Early19WeakSeparationMatroidSubdivision} and \cite{Early2019PlanarBasis}.

We ask whether there is a systematic construction of residues of $m^{(k)}_n$ which factorize in a meaningful way according to some notion of factorization channels.  Factorization in the usual sense means, roughly speaking, according to a 2-block set partition of the labels $\{1,\ldots, n\}$.  This seems to be a quite difficult problem due to the intrinsic complexity of the tropical Grassmannian; nonetheless the simplicity and the explicit nature of our all $(k,n)$ construction shows that the problem might just be tractable after all.  So our very first question is just to ask, are there general residues of the form 
\begin{eqnarray}\label{eq: factorization kn}
	\text{Res}_{\bullet=0}\left(m^{(k)}_n\right) & = & \prod_{j=1}^k m^{(k)}_{n_j}?
\end{eqnarray}
	\begin{figure}[h!]
	\centering
	\includegraphics[width=.85\linewidth]{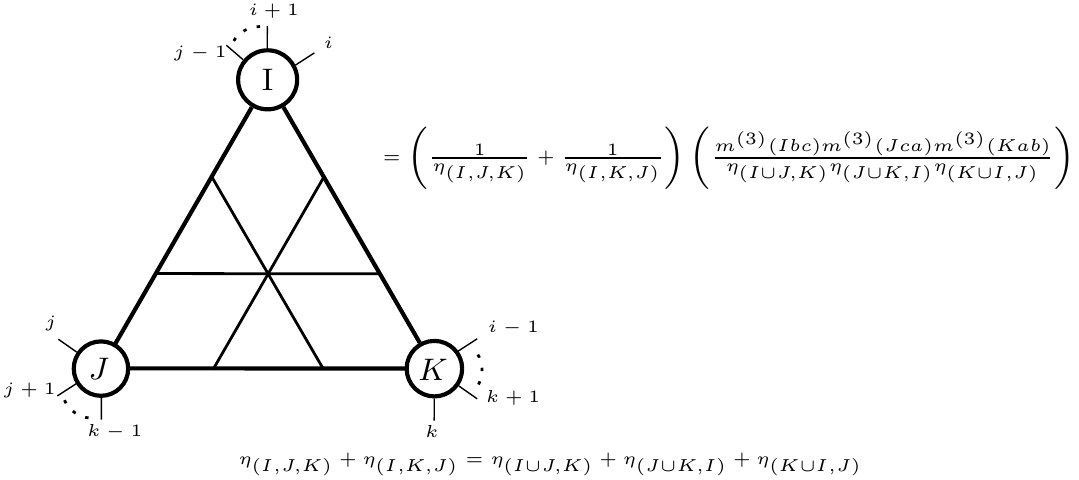}
	\caption{3-ary factorizations of CEGM amplitudes $m^{(3)}_n$.  Projection of a subdivision of the hypersimplex, via $\Delta_{3,n}$: $x\mapsto \left(\sum_{i\in I}x_i,\sum_{j\in J}x_j,\sum_{k\in K}x_k\right)$.}
	\label{fig:3n channels Intro}
\end{figure}
And if so, where are they?  Could there be some combinatorial representation like Feynman diagrams, where such factorization channels are manifest?

In Figure \ref{fig:3n channels Intro} the five linear functions of the form $\eta_{(\bullet)}$ explicitly in the denominator are determined by the CEGM scattering equations formula and can be calculated explicitly using a purely combinatorial formula for regular matroid subdivisions, given in Definition \ref{defn: kinematic blade A}.

Our solution to these problems, formulated for all $(k,n)$ and detailed in the cases $k=3,4$, finds that factorization of $m^{(k)}_n$ occurs after taking not one residue, but generically after $(k-1)^2$ of them.  Including degenerate cases then one has to take a $d(k-1)$-dimensional residue where $2\le d\le k$.  We remark that the analog of our formula in the case $k=2$ is the standard tree-level relation  (for the biadjoint scalar $m^{(2)}_n$, in particular)
$$\text{Res}_{P^2=0}(m^{(2)}_n(S\cup S^c)) = m^{(2)}(S,i)\cdot m^{(2)}(S^c,i),$$
where $P = \sum_{s\in S}p_s$ is a sum of momenta indexed by the set $S$.  Intriguingly, the $(k\ge 3)$-ary residues which we find are not in general governed by the usual poset of ``2-split'' factorizations as one might naively expect.  Let us point out that an interesting approach to combinatorial factorization was initiated in \cite{Cachazo2017} in the context of certain non-planar MHV leading singularities.

This paper completes the question asked in work of Cachazo, Early and Umbert \cite{CEU2021}, whether smooth 3-splits of $m^{(2)}_n$ provide a ``shadow'' of factorization for CEGM amplitudes.  We show here this intuition is the right one. 

This paper has two main sections and is structured as follows.  

In Section \ref{sec: factorization proposal} we formulate our main proposal to construct residues of $m^{(k)}_n$ which generically factorize into $k$ parts; the cases $k=3,4$ are studied in depth.  While for $k=3$ our proposal calls for a unique generic\footnote{Here generic means that one has a $k$-block set partition where all blocks have size at least one.} combinatorial type of residue, it is surprising that for $k=4$ we find two combinatorial types!  These two types are realized by two 3-dimensional polytopes in Figures \ref{fig:polypositroid4splita I} and \ref{fig:polypositroid4splitb II}; these graphs are dual to two positroidal subdivisions of the hypersimplex $\Delta_{4,n}$ which we construct explicitly.

In Section \ref{sec: wrap up}, we present a formula which has been very useful in studying coarsest positroidal subdivisions of hypersimplices $\Delta_{k,n}$ for small $k = 3,4$ especially.  However, we conjecture that works in general and will lead to new insights into simple poles of generalized biadjoint amplitudes $m^{(k)}_n$ and we hope that it will find applications well beyond.
\begin{figure}[h!]
	\centering
	\includegraphics[width=0.3\linewidth]{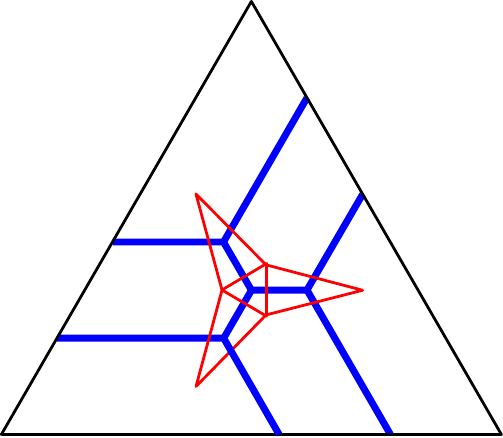}
	\caption{Matroidal weighted blade arrangement (thick, blue lines) and its dual (thin, red lines).  This is a section of a certain coarsest positroidal subdivision of $\Delta_{3,12}$; it has six maximal cells and is induced by the positive tropical Pl\"{u}cker vector $-\mathfrak{h}_{1,5,9}+\mathfrak{h}_{1,5,10}+\mathfrak{h}_{1,6,9}+\mathfrak{h}_{2,5,9}$.  Then $-\eta_{1,5,9} + \eta_{1,5,10} + \eta_{1,6,9} + \eta_{2,5,9} = 0$, where $\eta_{a,b,c}$ is an element of the planar basis of kinematic blades \cite{Early2019PlanarBasis}, is a simple pole of $m^{(3)}_{12}$. }
	\label{fig:39sixchambersubdivision Intro}
\end{figure}

\section{Factorization Proposal}\label{sec: factorization proposal}
We first establish some general definitions and conventions.

Denote by $\binom{\lbrack n\rbrack}{k}$ the set of all $k$-element subsets of $\lbrack n\rbrack = \{1,\ldots, n\}$, and by $\binom{\lbrack n\rbrack}{k}^{nf}$ the nonfrozen $k$-element subsets, excluding cyclic intervals of the form $\{j,j+1,\ldots, j+k-1\}$. 

With $p_J(g)$ the Pl\"{u}cker coordinate on $G(k,n)$ indexed by the $k$-element column set $J \subset \{1,\ldots, n\}$, denote by $X(k,n)$ the configuration space 
$$X(k,n) = \left\{g\in G(k,n): \prod_J p_J(g)\not=0\right\}\slash (\mathbb{C}^\ast)^n.$$
Denote by $\mathbf{N}_{k,n}$ the Newton polytope 
$$\mathbf{N}_{k,n} = \text{Newt}\left(\prod_J p_J\right)$$
evaluated on the matrix $M$; note that this provides a positive parametrization of $X(k,n)$, see Appendix \ref{sec:positive parametrization}. 
\begin{defn}
	The \textit{lineality subspace} $\text{Lin}_{k,n}$ is the image of the embedding $\mathbb{R}^n \hookrightarrow\mathbb{R}^{\binom{n}{k}}$, with coordinates
	$$(\pi)_J = \sum_{j\in J} x_j.$$
	That is,
	$$\text{Lin}_{k,n} = \text{span}\left\{\sum_{J \ni j} e^{J},\ j=1,\ldots, n\right\}.$$
\end{defn}
\begin{defn}[\cite{SpeyerWilliams2003}]
	Let $\pi \in \mathbb{R}^{\binom{n}{k}}$.  The point $\pi$ is positive tropical Pl\"{u}cker vector if, for each $(L,\{a,b,c,d\})$ with $L\cap\{a,b,c,d\} = \emptyset$, and $a<b<c<d$ up to cyclic rotation, we have
	$$\pi_{Lac} + \pi_{Lbd} = \min\{\pi_{Lab} + \pi_{Lcd},\pi_{Lad} + \pi_{Lbc}\}.$$
	The positive tropical Grassmannian\footnote{This definition relies on the following result: the positive Dressian is equal to the positive tropical Grassmannian \cite{arkani2020positive,SpeyerWilliams2020}.}  $\text{Trop}^+G(k,n) \subset \mathbb{R}^{\binom{n}{k}}\slash \text{Lin}_{k,n}$ is the set of all positive tropical Pl\"{u}cker vectors, modulo lineality.
\end{defn}

We now move to more specific constructions, starting with the planar basis of kinematic invariants $\eta_J(\mathfrak{s})$, introduced in \cite{Early2019PlanarBasis}.  The kinematic space $\mathcal{K}(k,n)$ is the codimension $n$ subspace of $\mathcal{K}(k,n)$ defined by 
$$\mathcal{K}(k,n) = \left\{(\mathfrak{s}) \in \mathbb{R}^{\binom{n}{k}}: \sum_{J\ni n} \mathfrak{s}_J = 0\right\}.$$
Following \cite{Early2019PlanarBasis}, for any $J\in \binom{\lbrack n\rbrack}{k}^{nf}$, define a linear function on the kinematic space $\eta_J:\mathcal{K}(k,n) \rightarrow \mathbb{R}$, by
\begin{eqnarray}\label{eq:planar basis element}
	\eta_J(\mathfrak{s}) & = & -\frac{1}{n}\sum_{I\in \binom{\lbrack n\rbrack}{k}}\min\{L_1(e_I-e_J),\ldots, L_n(e_I-e_J)\}\mathfrak{s}_I,
\end{eqnarray}
where 
$$L_j(x) = x_{j+1}+2x_{j+2} + \cdots +(n-1)x_{j-1},$$
for $j=1,\ldots, n$, are linear functions on $\mathbb{R}^n$.

Further define an element of $\mathbb{R}^{\binom{n}{k}}$, a height function $\mathfrak{h}_J$, by 
$$\mathfrak{h}_J = -\frac{1}{n}\sum_{I\in \binom{\lbrack n\rbrack}{k}}\min\{L_1(e_I-e_J),\ldots, L_n(e_I-e_J)\}e^I.$$

\begin{defn}
	A decorated ordered set partition $$((S_1)_{r_1},\ldots, (S_\ell)_{r_\ell})$$ of $(\{1,\ldots, n\},k)$ is an ordered set partition $(S_1,\ldots, S_\ell)$ of $\{1,\ldots, n\}$ together with an ordered partition $(r_1,\ldots, r_\ell)$ with $\sum_{j=1}^\ell r_j=k$.  It is said to be of type $\Delta_{k,n}$ if we have additionally $1\le r_j\le\vert S_j\vert-1 $, for each $j=1,\ldots, \ell$.  In this case we write $((S_1)_{r_1},\ldots, (S_\ell)_{r_\ell}) \in \text{OSP}(\Delta_{k,n})$.
\end{defn}
		
Given a decorated ordered set partition $(\mathbf{S},\mathbf{r})$, define 
\begin{eqnarray}\label{eq: planar basis dOSP A}
	M_{(\mathbf{S},\mathbf{r})_j}(x) & = & r_{j+1}x_{S_{j+1}} + (r_{j+1} + r_{j+2})x_{S_{j+1}\cup S_{j+2}} + \cdots (r_{j+1} + \cdots + r_{j-1})x_{S_{j+1} \cup \cdots \cup S_{j-1}}
\end{eqnarray}
for $j=1,\ldots, d$, where index addition is cyclic modulo $d$.

The following piecewise-linear function does the heavy-lifting in our story:
$$\rho_{(\mathbf{S},\mathbf{r})}(x) = \min\{M_{(\mathbf{S},\mathbf{r})_1}(x),\ldots,M_{(\mathbf{S},\mathbf{r})_d}(x)\}.$$
Here we point out that $\rho_{(\mathbf{S},\mathbf{r})}$ is manifestly invariant under cyclic block rotation.
\begin{defn}\label{defn: kinematic blade A}
	For any decorated ordered set partition 
	$$(\mathbf{S},\mathbf{r}) = ((S_1)_{r_1},\ldots, (S_d)_{r_d})$$
	of type $\Delta_{k,n}$, define a linear function on the kinematic space, the \textit{kinematic blade}
	$$\eta_{(\mathbf{S},\mathbf{r})}(\mathfrak{s}) = -\frac{1}{d}\sum_{J \in \binom{\lbrack n\rbrack}{k}} \rho_{(\mathbf{S},\mathbf{r})}(e_J)\mathfrak{s}_J$$
	and the height function
	$$\mathfrak{h}_{(\mathbf{S},\mathbf{r})}(\mathfrak{s}) = -\frac{1}{d}\sum_{J \in \binom{\lbrack n\rbrack}{k}} \rho_{(\mathbf{S},\mathbf{r})}(e_J)e^J.$$
\end{defn}

\subsection{Case $k=3$}

We represent in Figure \ref{fig:k3higherpropa A} the factorization cone $\mathcal{C}^{(3)}$ which, we conjecture, gives rise to a generic factorization of $m^{(3)}_n$ into three parts.  We use matroidal blade arrangements \cite{Early19WeakSeparationMatroidSubdivision}.  We note that this representation of the positive tropical Grassmannian is novel and different from other interpretations; instead of using tropical linear spaces or heights over the vertices of the hypersimplex to project regions of linearity onto the maximal cells, we now use weights attached to certain internal faces of the subdivision.  
Here $\mathcal{C}^{(3)}$ is identified abstractly with 
\begin{eqnarray*}
	\mathcal{C}^{(3)} & = & \left\{(c_{i,j}) \in (\mathbb{R}_{\ge 0})^6: (\ast)\right\},
\end{eqnarray*}
where $(\ast)$ consists of the three balancing equations (compare with \cite{EarlyBlades} and \cite[Section 2.5]{BC2019})
\begin{eqnarray}\label{eq: balancing equations A}
	c_{1,2} + c_{1,3} & = & c_{2,1} + c_{3,1}\nonumber\\
	c_{2,3} + c_{2,1} & = & c_{3,2} + c_{1,2}\\
	c_{3,1} + c_{3,2} & = & c_{1,3} + c_{2,3}\nonumber.
\end{eqnarray}
See Figure \ref{fig:k3higherpropa A}.
\begin{figure}[!h]
	\centering
	\includegraphics[width=1\linewidth]{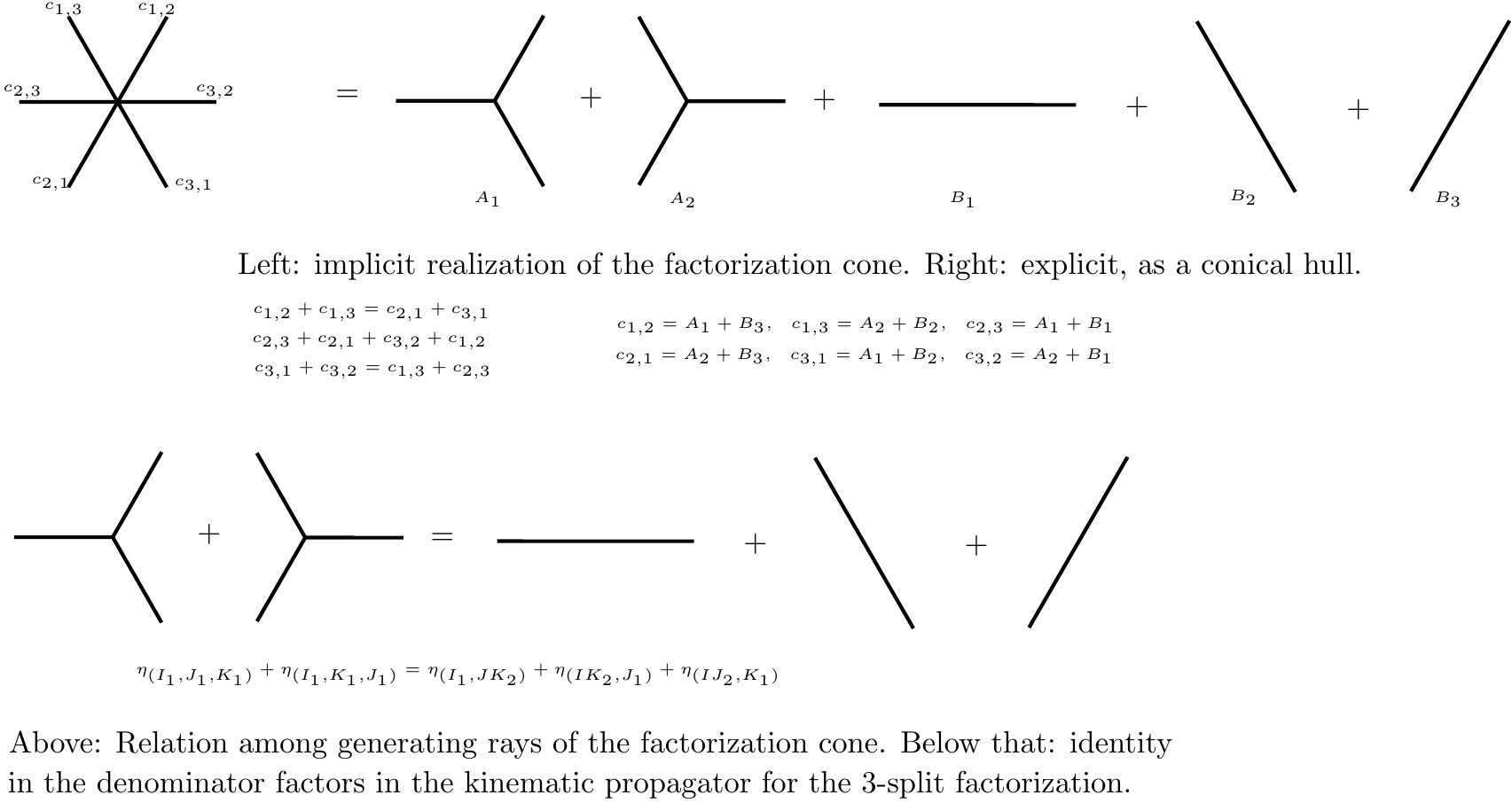}
	\caption{Blade arrangement interpretation of the generating rays of the factorization cone $\mathcal{C}^{(3)}$; this can be embedded into the positive tropical Grassmannian $\text{Trop}^+G(3,n)$.  Parameters $u_1,\ldots, u_6,A_i,B_j$ are nonnegative numbers; the $u_j$ satisfy the indicated balancing equations.}
	\label{fig:k3higherpropa A}
\end{figure}
For example, to be completely explicit, one of the bipyramidal cones in $\text{Trop}^+G(3,6)$ can be represented of all formal linear combinations of matroid polytopes
\begin{eqnarray}\label{eq: formal linear combination blades Minkowski sum}
	& & c_{1,2}\lbrack 56_1\rbrack \boxplus \lbrack  12_1 34_1\rbrack + c_{2,1}\lbrack 56_1\rbrack \boxplus \lbrack  34_1 12_1\rbrack +  c_{2,3}\lbrack 12_1\rbrack \boxplus \lbrack  34_1 56_1\rbrack + c_{3,2}\lbrack 12_1\rbrack \boxplus \lbrack  56_1 34_1\rbrack\\
	& + &  c_{3,1}\lbrack 34_1\rbrack \boxplus \lbrack  56_1 12_1\rbrack + c_{3,1}\lbrack 34_1\rbrack \boxplus \lbrack  12_1 56_1\rbrack\nonumber
\end{eqnarray}
such that the balancing conditions of Equation \eqref{eq: balancing equations A} hold with all $c_{i,j}\ge 0$.  Here we note the defining equations, for example
$$((12_1 34_1 56_1)) = \lbrack 12_1\rbrack \boxplus \lbrack 34_1 56_1\rbrack + \lbrack 34_1\rbrack \boxplus \lbrack 56_1 12_1\rbrack+\lbrack 56_1\rbrack \boxplus \lbrack  12_1 34_1\rbrack$$
and
$$((12_1 3456_2)) = \lbrack 12_1 \rbrack \boxplus \lbrack 3456_2 \rbrack,$$
where these satisfy the one relation 
$$((12_1 34_1 56_1)) + ((12_1 56_1 34_1)) = ((12_1 3456_2)) + ((34_1 5612_2))+((56_1 1234_2)).$$
Here the symbol $\boxplus$ denotes the Minkowski sum; in this way each term in Equation \eqref{eq: formal linear combination blades Minkowski sum} is a Minkowski sum of a line segment and a half octahedron.

We now compute some f-vectors; first, the factorization cone $\mathcal{C}^{(3)}$ above has f-vector $(5,9,6,1)$.  When $n=6$ this is one of the bipyramidal cones in $\text{Trop}^+G(3,6)$.

For sake of comparison, omitting details, we report the results of additional calculations of f-vectors for certain new $(4,8)$ and $(5,10)$ factorization cones, which can be defined analogously to $\mathcal{C}^{(3)}$ using the data from Section \ref{sec: all kn proposal}.  For $n=4$ and $n=5$ we find 18 (given in Equation \eqref{eq: 18 propagators 48 A}) and 63 rays, respectively:
$$(18, 108, 308, 485, 450, 250, 81, 14, 1)$$
and
$$( 63, 895, 6010, 23965, 63191, 116936, 157285, 156950, 117405, 65985, 27704, 8555, 1885, 280, 25, 1).$$
To obtain these f-vectors, we took an arbitrary linear combination of 9 (respectively 16) blades and then took two (respectively three) boundaries.  Then we defined the cone in $\mathbb{R}^{9}$ (respectively $\mathbb{R}^{16}$) using a system of inequalities that arises by requiring all coefficients in the respective boundaries be nonnegative, which is the condition under which weighted blade arrangements are positroidal.  The f-vectors were calculated using SageMath.

Turning now to generalized biadjoint scalar amplitudes, we conjecture an explicit formula for 3-ary factorized residues, when $k=3$.
\begin{conjecture}\label{conjecture bipyramidal factorization A}
	Fix an ordered set partition $(I,J,K)$ of $\{1,\ldots, n\}$, where
	$$I = \{i,i+1,\ldots, j-1\},\ J = \{j,j+1,\ldots, k-1\},\ K = \{k,k+1,\ldots, i-1\}.$$
	Then to leading order in $\eta_{(I_1,JK_2)},\eta_{(IJ_2,K_1)},\eta_{(KI_2,J_1)}$ we have
	\begin{eqnarray}\label{eq: residue 3n splitting A}
		m^{(3)}_n & = & \left(\frac{1}{\eta_{(I_1,J_1,K_1)}}+\frac{1}{\eta_{(I_1,J_1,K_1)}}\right)\cdot \left(\frac{m^{(3)}(Ibc)m^{(3)}(Jac)m^{(3)}(Kab)}{\eta_{(IJ_2,K_1)}\eta_{(KI_2,J_1)} \eta_{(I_1,JK_2)}}\right),
	\end{eqnarray}
	where there is one linear relation among the five kinematic blades which appear explicitly, namely
	$$\eta_{(I_1,J_1,K_1)}+\eta_{(I_1,K_1,J_1)} =\eta_{(I_1,JK_2)}+\eta_{(IJ_2,K_1)}+\eta_{(KI_2,J_1)}. $$
\end{conjecture}
	\begin{figure}[h!]
	\centering
	\includegraphics[width=.9\linewidth]{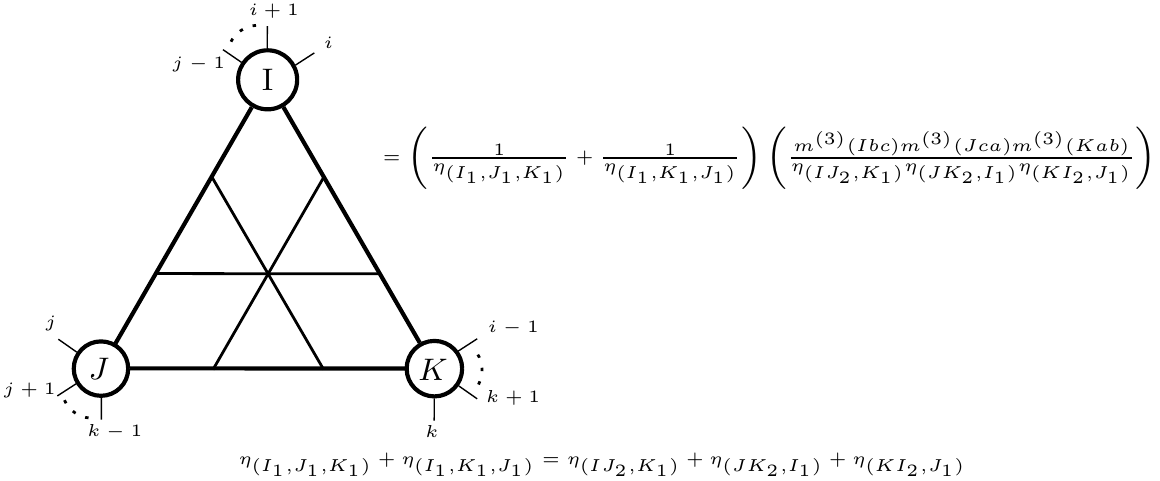}
	\caption{3-ary factorizations of CEGM amplitudes $m^{(3)}_n$.  Projection of a subdivision of the hypersimplex, via $\Delta_{3,n}$: $x\mapsto \left(\sum_{i\in I}x_i,\sum_{j\in J}x_j,\sum_{k\in K}x_k\right)$.}
\end{figure}
\begin{example}
	Let us give the complete formulation of the prefactors in Conjecture \ref{conjecture bipyramidal factorization A} when $n = 6,7,8$.
	
	Modulo cyclic permutation there a small enough number of cases that the computation can be performed in full by calculating the explicit formulas for the CEGM amplitudes and then taking residues; we list representatives of the prefactors modulo cyclic relabeling: for $m^{(3)}_6$ there is a single case,
	$$\left(\frac{1}{\eta_{(12_1 34_1 56_1)}} + \frac{1}{\eta_{(12_1 56_1 34_1)}}\right).$$
	For $m^{(3)}_7$ there is again a single case,
	$$\left(\frac{1}{\eta_{(12_1 34_1 567_1)}} + \frac{1}{\eta_{(12_1 567_1 34_1)}}\right),$$
	while for $m^{(3)}_8$ there are two cases,
	$$\left(\frac{1}{\eta_{(12_1 34_1 5678_1)}} + \frac{1}{\eta_{(12_1 5678_1 34_1)}}\right)\text{ and } \left(\frac{1}{\eta_{(12_1 345_1 678_1)}} + \frac{1}{\eta_{(12_1 678_1 345_1)}}\right).$$
	In the bottom line, comparing to Equation \eqref{eq: residue 3n splitting A} we identify respectively $m^{(3)}_6\cdot m^{(3)}_4\cdot m^{(3)}_4$, while in the second case the prefactor multiplies $m^{(3)}_5\cdot m^{(3)}_5\cdot m^{(3)}_4$, where we remind that $m^{(3)}_4=1$. 
\end{example}
\subsection{More with $k=3$: A Cubical Fibration}

Let $J = \{j_1,j_2,j_3\} \in \binom{\lbrack n\rbrack}{3}^{nf}$ be totally nonfrozen, that is with no two indices cyclically adjacent.  Then, for each row in the table below, the four $\gamma$'s are indexed by a weakly separated collection and are consequently simultaneously minimized on a face of the Newton polytope $\mathbf{N}_{3,n}$.  These eight faces are given by simultaneously minimizing on $\mathbf{N}_{3,n}$ the generalized positive roots $\gamma_{j_1,j_2,j_3}$ in each of the eight rows, respectively
\begin{eqnarray}\label{eq: eight three splits}
	\begin{array}{cccc}
		\gamma _{j_1,j_2,j_3} & \gamma _{j_1,j_1+1,j_3} & \gamma _{j_1,j_2,j_2+1} & \gamma _{j_2,j_3,j_3+1} \\
		\gamma _{j_1,j_2,j_3} & \gamma _{j_1,j_2,j_1-1} & \gamma _{j_1,j_2,j_2+1} & \gamma _{j_2,j_3,j_3+1} \\
		\gamma _{j_1,j_2,j_3} & \gamma _{j_1,j_3-1,j_3} & \gamma _{j_1,j_1+1,j_3} & \gamma _{j_1,j_2,j_2+1} \\
		\gamma _{j_1,j_2,j_3} & \gamma _{j_1,j_3-1,j_3} & \gamma _{j_1,j_2,j_1-1} & \gamma _{j_1,j_2,j_2+1} \\
		\gamma _{j_1,j_2,j_3} & \gamma _{j_2-1,j_2,j_3} & \gamma _{j_1,j_1+1,j_3} & \gamma _{j_2,j_3,j_3+1} \\
		\gamma _{j_1,j_2,j_3} & \gamma _{j_2-1,j_2,j_3} & \gamma _{j_1,j_2,j_1-1} & \gamma _{j_2,j_3,j_3+1} \\
		\gamma _{j_1,j_2,j_3} & \gamma _{j_2-1,j_2,j_3} & \gamma _{j_1,j_3-1,j_3} & \gamma _{j_1,j_1+1,j_3} \\
		\gamma _{j_1,j_2,j_3} & \gamma _{j_2-1,j_2,j_3} & \gamma _{j_1,j_3-1,j_3} & \gamma _{j_1,j_2,j_1-1}. \\
	\end{array}
\end{eqnarray}
Here with $k=3$, the generalized positive root $\gamma_{j_1,j_2,j_3}$ is a particular linear function on $\mathbb{R}^{(k-1)\times (n-k)}$, defined by 
$$\gamma_{j_1,j_2,j_3} = \sum_{t=j_1}^{j_2-2}\alpha_{1,t} + \sum_{t=j_2-1}^{j_3-3}\alpha_{2,t}.$$
In general, generalized positive roots are defined by the equation 
$$\gamma_J(\alpha) = \sum_{i=1}^{k-1} \alpha_{i,\lbrack j_i-(i-1),j_{i+1}-i-1\rbrack},$$
for any subset $J \in \binom{\lbrack n\rbrack}{k}$, and may be visualized on a grid as in Figure \ref{fig:higher-root-path}.  See \cite{CE2020B} for the original definition of generalized positive roots and \cite{E2021} for details with more connections and applications to matroid subdivisions, triangulations of generalized root polytopes and the noncrossing complex $\mathbf{NC}_{k,n}$, introduced in Section \ref{sec: all kn proposal}.
\begin{figure}[h!]
	\centering
	\includegraphics[width=0.55\linewidth]{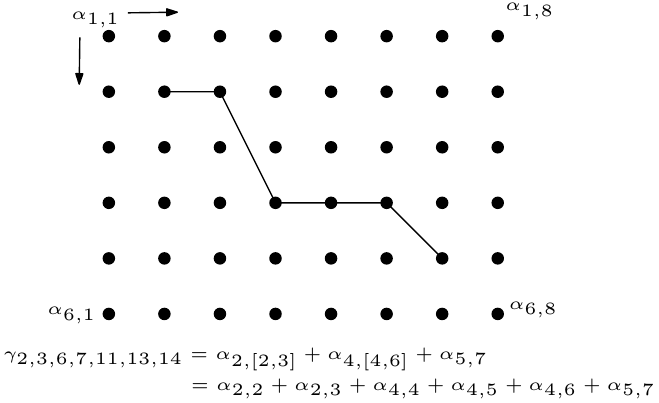}
	\caption{Staircase representation of a generalized positive root.}
	\label{fig:higher-root-path}
\end{figure}

\begin{conjecture}\label{conjecture 3-split 3n}
	The facet of $\mathbf{N}_{3,n}$ which minimizes $\gamma_{j_1,j_2,j_3}$, with $\{j_1,j_2,j_3\}$ totally nonfrozen, has exactly eight (codimension 3) faces which are combinatorially isomorphic to 
	$$\mathbf{N}_{3,j_2-j_1+2} \times \mathbf{N}_{3,j_3-j_2+2}\times \mathbf{N}_{3,j_1-j_3+n+2}.$$
\end{conjecture}
\begin{rem}
	Let us point out the following feature: the sum of the products of the reciprocals factors.  Namely, we find
	\begin{eqnarray}
		\frac{1}{\gamma_{j_1,j_2,j_3}}\left(\frac{1}{\gamma _{j_1,j_2,j_1-1}}+\frac{1}{\gamma _{j_1,j_1+1,j_3}}\right) \left(\frac{1}{\gamma _{j_2-1,j_2,j_3}}+\frac{1}{\gamma _{j_1,j_2,j_2+1}}\right) \left(\frac{1}{\gamma _{j_2,j_3,j_3+1}}+\frac{1}{\gamma _{j_1,j_3-1,j_3}}\right),
	\end{eqnarray}
	which clearly generalizes the residue of $m^{(3)}_6$ corresponding to the pole ``$R=0$'' in \cite{CEGM2019}.
\end{rem}

\subsection{Case $k=4$}

For $k=4$ and $n\ge 8$, say, given an ordered set partition $\mathbf{S} = (S_1,S_2,S_3,S_4)$ of $\{1,\ldots, n\}$ with all blocks of size at least two, then there are two combinatorially distinct sets of propagators for $m^{(4)}_n$, type I and type II, say.  Let us be concrete.  For simplicity we list the propagators only when $(k,n) = (4,8)$ and $\mathbf{S} = (18,23,45,67)$; larger $n$ cases may be obtained by substitution, as discussed next.
\begin{rem}
	  In what follows, generic larger $n$ examples of 4-ary factorizations of $m^{(4)}_n$ are obtained from the following table by making the replacement $j \mapsto T_j$ on the right-hand side, where $(T_1,\ldots, T_8)$ is an ordered set partition of $\{1,\ldots, n\}$:
	\begin{eqnarray*}
		\eta_{1,2,3,5} & = & \eta_{(123678_3 45_1)}\\
		\eta_{1,3,5,7} & = & \eta_{(18_1 23_1 45_1 67_1)}\\
		\eta_{1,3,5,8} & = & \eta_{(1678_2 23_1 45_1)}\\
		\eta_{1,3,6,7} & = & \eta_{(18_1 23_1 4567_2)}\\
		\eta_{1,3,7,8} & = & \eta_{(145678_3 23_1)}\\
		\eta_{1,4,5,7} & = & \eta_{(18_1 2345_2 67_1)}\\
		\eta_{1,5,6,7} & = & \eta_{(18_1 234567_3)}\\
		\eta_{2,3,5,7} & = & \eta_{(1238_2 45_1 67_1)}\\
		\eta_{2,3,6,7} & = & \eta_{(1238_2 4567_2)}\\
		\eta_{3,4,5,7} & = & \eta_{(123458_3 67_1)}\\
		\eta_{1,4,5,8} & = & \eta_{(1678_2 2345_2)}.
	\end{eqnarray*}
\end{rem}
	\begin{figure}[h!]
	\centering
	\includegraphics[width=0.6\linewidth]{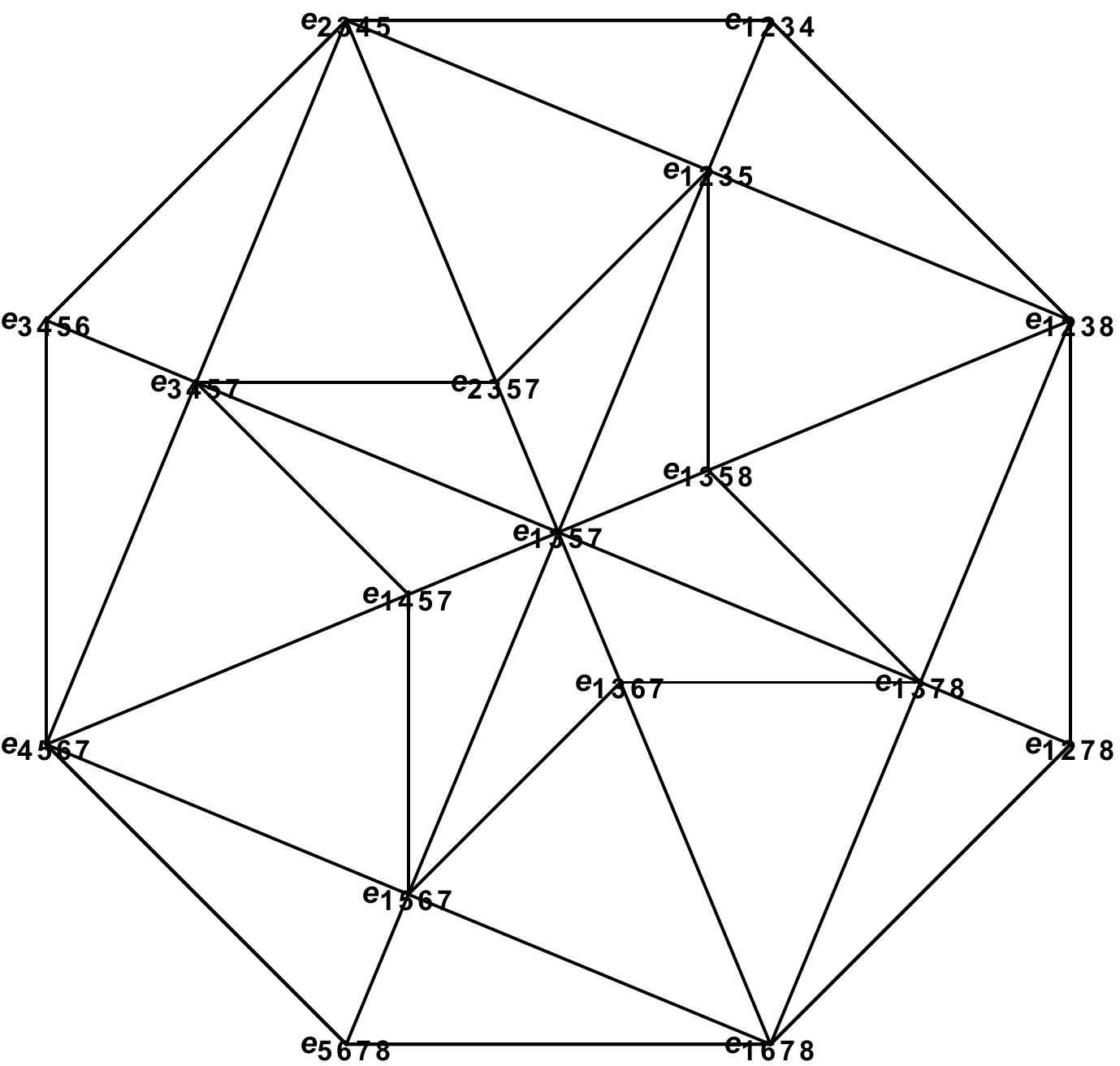}
	\caption{Locations on $\Delta_{4,8}$ of the nine kinematic blades for the 4-ary factorization of type I.}
	\label{fig:4-mass-box-48-plabic-graph A}
\end{figure}

Then the full list of propagators for type I factorization channel associated to $\mathbf{S}$ is
\begin{eqnarray}\label{eq: 18 propagators 48 A}
	\begin{array}{cc}
		\eta _{1,2,3,5} & -\eta _{1,3,5,7}+\eta _{1,2,3,5}+\eta _{1,3,6,7}+\eta _{1,4,5,7} \\
		\eta _{1,3,5,7} & -\eta _{1,3,5,7}+\eta _{1,3,5,8}+\eta _{1,3,6,7}+\eta _{1,4,5,7} \\
		\eta _{1,3,5,8} & -\eta _{1,3,5,7}+\eta _{1,3,5,8}+\eta _{1,3,6,7}+\eta _{2,3,5,7} \\
		\eta _{1,3,6,7} & -\eta _{1,3,5,7}+\eta _{1,3,5,8}+\eta _{1,4,5,7}+\eta _{2,3,5,7} \\
		\eta _{1,3,7,8} & -\eta _{1,3,5,7}+\eta _{1,3,6,7}+\eta _{1,4,5,7}+\eta _{2,3,5,7} \\
		\eta _{1,4,5,7} & -\eta _{1,3,5,7}+\eta _{1,3,7,8}+\eta _{1,4,5,7}+\eta _{2,3,5,7} \\
		\eta _{1,5,6,7} & -\eta _{1,3,5,7}+\eta _{1,3,5,8}+\eta _{1,5,6,7}+\eta _{2,3,5,7} \\
		\eta _{2,3,5,7} & -\eta _{1,3,5,7}+\eta _{1,3,5,8}+\eta _{1,3,6,7}+\eta _{3,4,5,7} \\
		\eta _{3,4,5,7} & -2 \eta _{1,3,5,7}+\eta _{1,3,5,8}+\eta _{1,3,6,7}+\eta _{1,4,5,7}+\eta _{2,3,5,7}.
	\end{array}
\end{eqnarray}
There is a corresponding subdivision of the hypersimplex $\Delta_{4,8}$, whose dual is depicted in Figure \ref{fig:polypositroid4splita I}.

\begin{figure}[h!]
	\centering
	\includegraphics[width=0.4\linewidth]{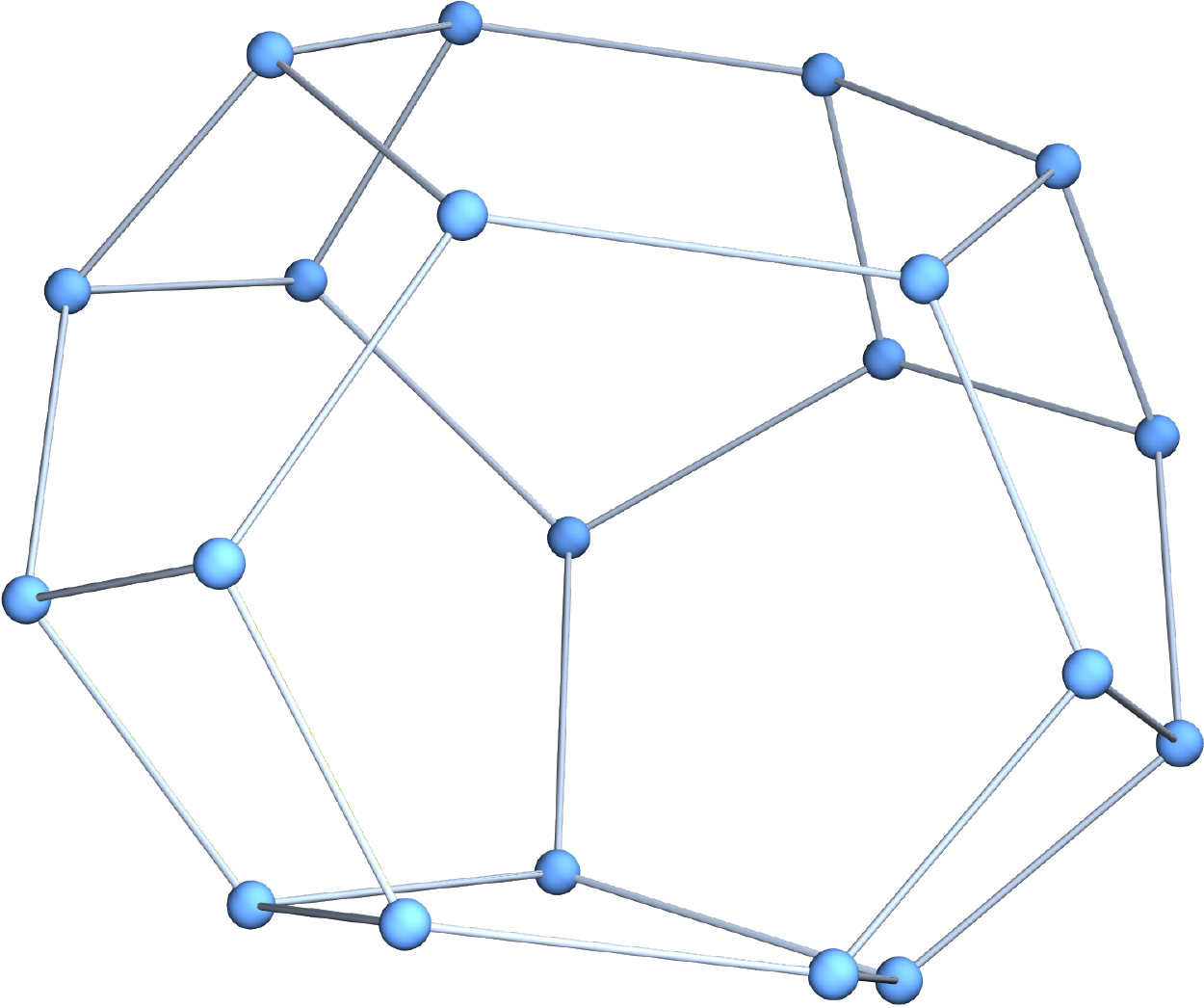}
	\caption{Type I factorization channel, represented as the (dual of the) first of two combinatorially non-isomorphic subdivisions of $\Delta_{4,n}$; this conjecturally always induces a 4-ary factorization of $m^{(4)}_n$.}
	\label{fig:polypositroid4splita I}
\end{figure}
\begin{figure}[h!]
	\centering
	\includegraphics[width=0.4\linewidth]{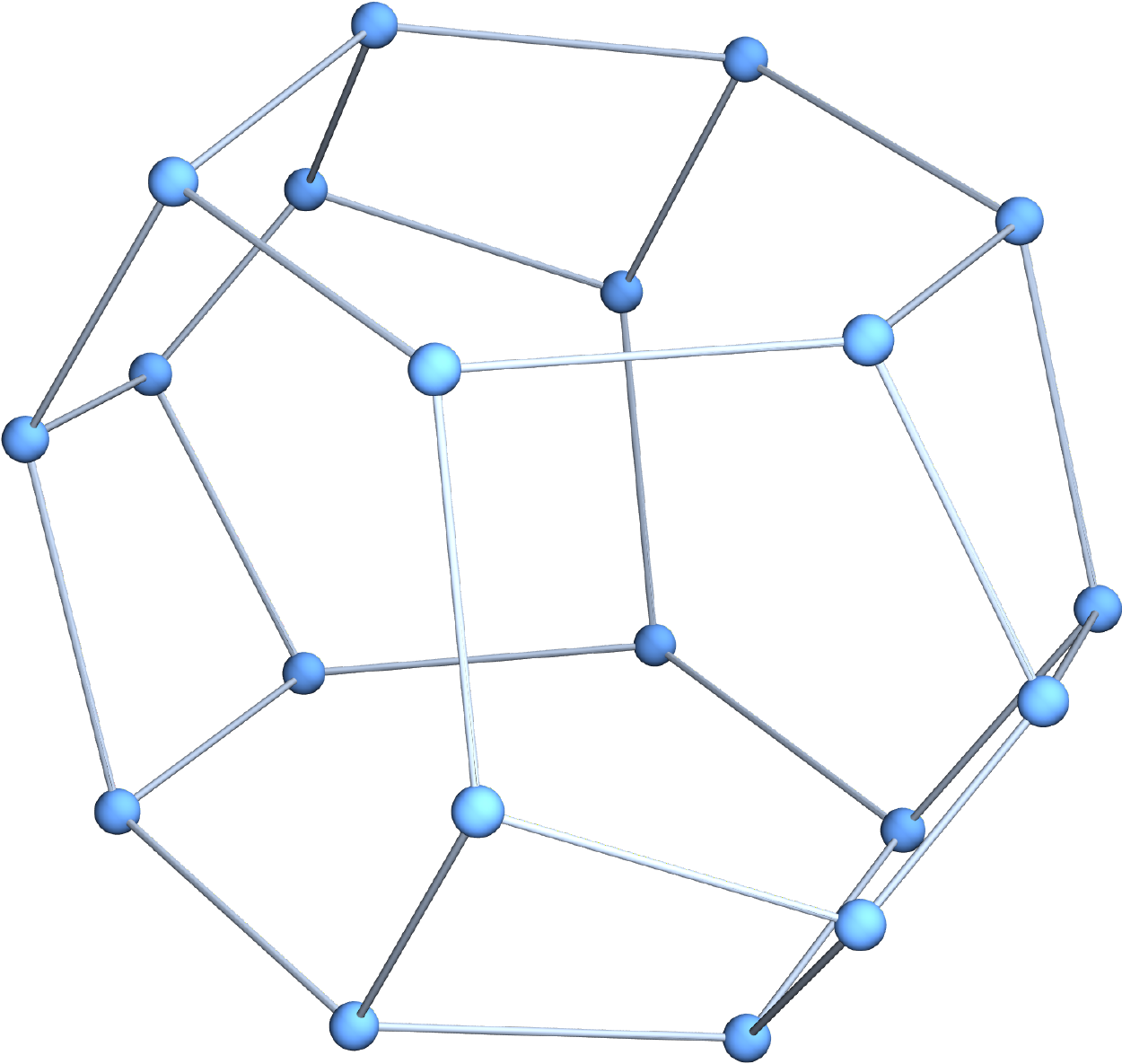}
	\caption{Type II factorization channel, the (dual of the) second of our two combinatorially inequivalent subdivisions of $\Delta_{4,n}$; this also conjecturally induces a four-fold factorization of $m^{(4)}_n$.  The collection of propagators is given in Equation \eqref{eq: 18 propagators A}.}
	\label{fig:polypositroid4splitb II}
\end{figure}
The full list of propagators for the type II factorization channel associated to $\mathbf{S}$ is 
\begin{eqnarray}\label{eq: 18 propagators A}
\begin{array}{cc}
	\eta _{1,2,3,5} & 	\eta _{1,2,3,5}-\eta _{1,3,5,7}+\eta _{1,3,6,7}+\eta _{1,4,5,7} \\
	\eta _{3,4,5,7} &
	-\eta _{1,3,5,7}+\eta _{1,3,5,8}+\eta _{1,3,6,7}+\eta _{1,4,5,7} \\
	\eta _{1,3,5,8} &	-\eta _{1,3,5,7}+\eta _{1,3,7,8}+\eta _{1,4,5,7}+\eta _{2,3,5,7} \\
	\eta _{1,3,7,8} &
	\eta _{1,2,3,5}-\eta _{1,3,5,8}+\eta _{1,3,7,8}+\eta _{1,4,5,8}\\
	\eta _{1,4,5,7} &
	-\eta _{1,3,5,7}+\eta _{1,3,5,8}+\eta _{1,5,6,7}+\eta _{2,3,5,7}\\
	\eta _{1,4,5,8} &
	-\eta _{1,4,5,7}+\eta _{1,4,5,8}+\eta _{1,5,6,7}+\eta _{3,4,5,7}\\
	\eta _{1,5,6,7} &
	-2 \eta _{1,3,5,7}+\eta _{1,3,5,8}+\eta _{1,3,6,7}+\eta _{1,4,5,7}+\eta _{2,3,5,7}\\
	-\eta _{1,3,5,7}+\eta _{1,3,5,8}+\eta _{1,3,6,7}+\eta _{3,4,5,7} &	\eta _{1,2,3,5}-\eta _{1,3,5,7}+\eta _{1,3,6,7}+\eta _{1,4,5,8}+\eta _{3,4,5,7} \\
	-\eta _{1,3,5,7}+\eta _{1,3,5,8}+\eta _{1,4,5,7}+\eta _{2,3,5,7} & 
	-\eta _{1,3,5,7}+\eta _{1,3,7,8}+\eta _{1,4,5,8}+\eta _{1,5,6,7}+\eta _{2,3,5,7}.
\end{array}
\end{eqnarray}

In Figures \ref{fig:polypositroid4splita I} and \ref{fig:polypositroid4splitb II} we present the dual graphs of the positroidal subdivisions of $\Delta_{4,8}$ for the corresponding two combinatorially inequivalent cones in $\text{Trop}G^+(4,n)$ which conjecturally both always induce 4-fold factorizations.
\begin{figure}
	\centering
	\includegraphics[width=0.85\linewidth]{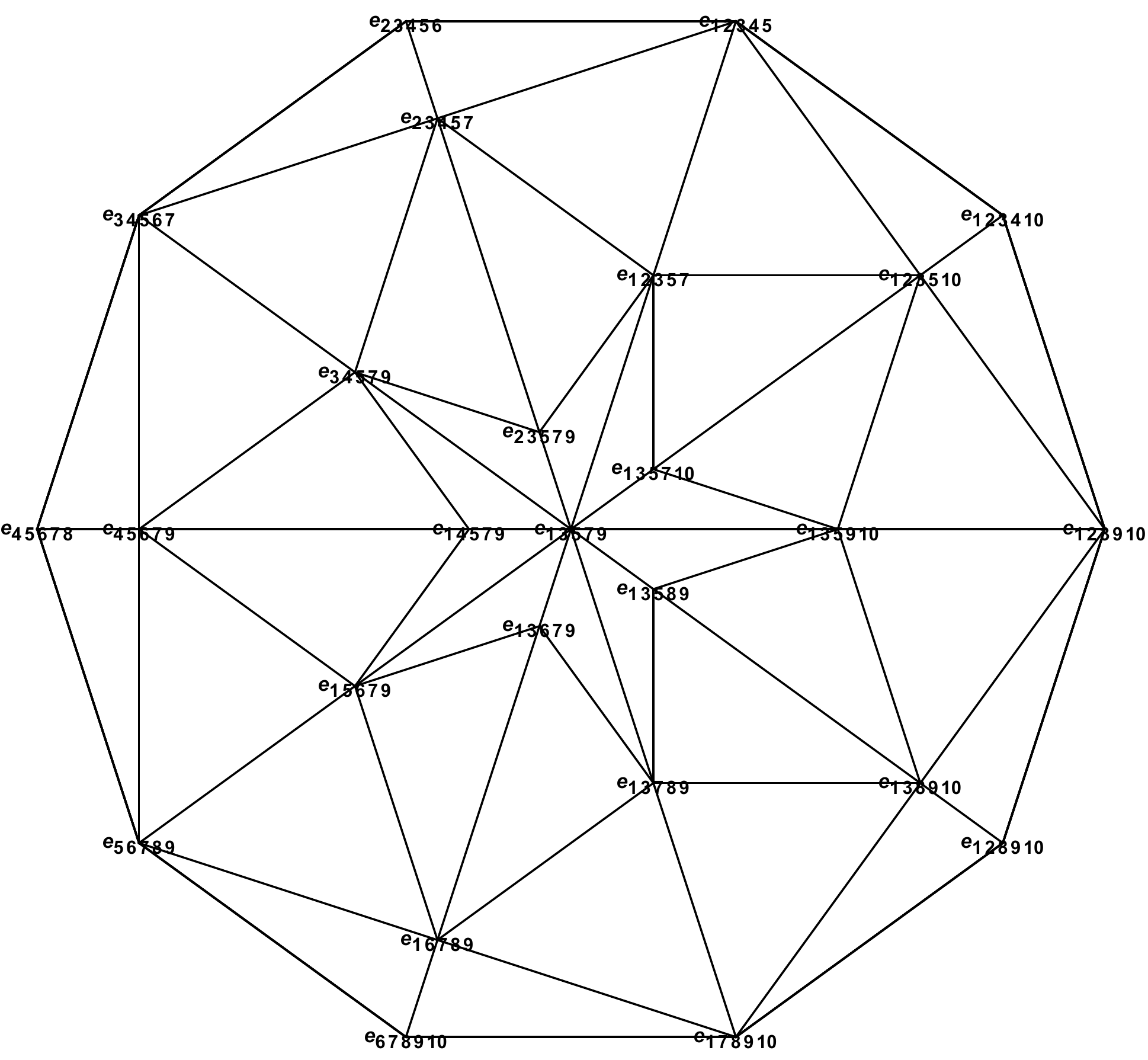}
	\caption{Locations on the hypersimplex $\Delta_{5,10}$ of the 16 kinematic blades in the kinematic propagator for (conjectural) 5-splits of $m^{(5)}_{n}$.  Vertices of the graph form a maximal weakly separated collection in $\mathbf{WS}_{5,10}$. A 5-split is induced by replacing the ordered set partition $\mathbf{S} = (10\ 1,23,45,67,89)$ with some ordered set partition $\mathbf{S} = (S_1,\ldots, S_5)$, using the decorated ordered set partition formula for kinematic blades.  Frozen vertices again included in the figure for clarity.}
	\label{fig:5splitpropagator510graph}
\end{figure}

Finally, in Figure \ref{fig:5splitpropagator510graph} we present the analog of Figure \ref{fig:4-mass-box-48-plabic-graph A}, but for $(k,n) = (5,10)$.  We leave to future work the problem to classify the types of propagators associated to a given set partition with all blocks of size at least 2, as above we did for $k=4$.

\subsection{An all (k,n) Proposal}\label{sec: all kn proposal}

In order to formulate our all (k,n) proposal, let us begin with some definitions and basic results.

 In \cite{LeclercZelevinsky} Leclerc and Zelevinsky introduced the notion of weak separation; for the noncrossing condition, see \cite{PKPS,Pylyavskyy2009} and in particular \cite{GrassmannAssociahedron}.

\begin{defn}
	A pair $I,J \in \binom{\lbrack n\rbrack}{k}$ is said to be \textit{weakly separated}, with respect to the cyclic order $(1,2,\ldots, n)$, provided that the coordinates in the difference  $e_I-e_J$ of vertices $e_I,e_J\in \Delta_{k,n}$ does not contain the pattern $e_a-e_b+e_c-e_d$ for $a<b<c<d$, up to cyclic rotation.
	
	A pair $I,J \in \binom{\lbrack n\rbrack}{k}$ of $k$-element subsets of $\{1,\ldots, n\}$ is said to be \textit{non-crossing}, with respect to the linear order $1<2<\cdots <n$, provided that for each $1\le a<b\le k$, then either
	\begin{enumerate}
		\item The pair $\{\{i_a,i_{a+1},\ldots, i_{b}\},\{j_a,j_{a+1},\ldots, j_b\}\}$ is weakly separated, or 
		\item The interiors of the respective intervals do not coincide, that is we have
		$$\{i_{a+1},\ldots, i_{b-1}\} \not= \{j_{a+1},\ldots, j_{b-1}\}.$$
	\end{enumerate}
\end{defn}
Clearly these is some redundancy in the definition of the noncrossing condition, but in our context it is convenient to keep it like this.

Denote by $\mathbf{WS}_{k,n}$ the poset of all collections of pairwise weakly separated \textit{nonfrozen} $k$-element subsets, ordered by inclusion.  Similarly, let $\mathbf{NC}_{k,n}$ be the poset of all collections of pairwise noncrossing \textit{nonfrozen} $k$-element subsets, ordered by inclusion.

\begin{thm}[\cite{Early19WeakSeparationMatroidSubdivision}]\label{thm: weakly separated matroid subdivision blade}
	Given a collection of vertices $e_{I_1},e_{I_2},\ldots, e_{I_m}\in\Delta_{k,n}$, the height function 
	$$\pi = \sum_{j=1}^m \mathfrak{h}_{I_j}$$
	induces a matroidal (in particular positroidal) subdivision if and only if the collection $\{I_1,\ldots, I_m\}$ is pairwise weakly separated.	
\end{thm}

\begin{cor}
	Given any weakly separated collection $\{I_1,\ldots, I_m\}$, then for all $c_{1},\ldots ,c_m \ge 0$ we have 
	$$\pi = \sum_{j=1}^m c_j\mathfrak{h}_{I_j} \in \text{Trop}^+G(k,n).$$
\end{cor}

Now let us fix an ordered set partition $\mathbf{S} = (S_1,\ldots, S_k)$ of $\{1,\ldots, n\}$.

Let $\mathbf{X}(\mathbf{S})$ be the collection of decorated ordered set partitions
\begin{eqnarray}\label{eq: central blade arrangement intro A}
	\left\{(((S_a)_1,(S_{a+1})_1,\ldots, (S_b)_1, (S_{b+1} \cup \cdots \cup S_{a-1})_{k-(b-a+1)})): a<b<b+1<a-1\right\}
\end{eqnarray}
where addition in the subscripts $j$ on the blocks $S_j$ is cyclic modulo $k$.  It is easy to verify that there are $(k-1)^2-1$ decorated ordered set partitions in this collection, having excluded $(\mathbf{S},(1,\ldots, 1))$ itself.

Denote by $\mathcal{N}_{\mathbf{S}} $ the subset $\mathcal{N}_{\mathbf{S}} \subseteq \mathbf{X}(\mathbf{S})$ that indexes the distinct and nonzero kinematic blades $\eta_{(\mathbf{T})}$, as can be calculated using the explicit formula in Definition \ref{defn: kinematic blade A}.  

\begin{rem}
	When all blocks in $\mathbf{S}$ have size at least two, then $\mathcal{N}_{\mathbf{S}}$ still has $(k-1)^2-1$ distinct elements.  More generally, if exactly $d$ blocks in $\mathbf{S}$ have size at least 2, then the size of $\mathcal{N}_{\mathbf{S}}$, together with $\mathbf{S}$, drops to $(d-1)(k-1)$.
\end{rem}
\begin{thm}\label{thm: cone in tropGrass A}
	Fix $(k,n)$ with $2\le k\le n-2$ and let $J \in \binom{\lbrack n\rbrack}{k}^{nf}$ be given.  Define an ordered set partition $\mathbf{S} = (S_1,\ldots, S_k)$, where 
	$$S_1 = \{j_{k}+1,\ldots, j_1\},\ \ S_2 = \{j_{1}+1,\ldots, j_2\},\ldots, S_k = \{j_{k-1}+1,\ldots, j_k\}$$
	modulo $n$.
	
	Then the following is a cone in the positive tropical Grassmannian:
	$$\left\{c_\mathbf{S} \mathfrak{h}_\mathbf{S} + \sum_{\mathbf{T} \in \mathcal{N}_\mathbf{S}}c_\mathbf{T} \mathfrak{h}_\mathbf{T}: c_\mathbf{S},c_\mathbf{T} \ge 0\right\} \subset \text{Trop}^+G(k,n).$$ 
	
\end{thm}

\begin{proof}
	Given $J$ as in the statement of the Theorem, applying Definition \ref{def: desp bijection} (see also \cite{Early19WeakSeparationMatroidSubdivision}) gives rise to a decorated ordered set partition 
	$$(\mathbf{T},\mathbf{r}) = ((T_1)_{r_1},\ldots, (T_k)_{r_k}) \in \text{OSP}(\Delta_{k,n})$$
	such that 
	$$((1,2,\ldots, n))_{e_J} \cap \Delta_{k,n} = ((T_1)_{r_1},\ldots, (T_k)_{r_k})\cap \Delta_{k,n}.$$
	Here the ordered set partition $(T_1,\ldots, T_d)$ has been formed from $(S_1,\ldots, S_k)$ by joining certain blocks as determined by intersecting the blade with the hypersimplex as indicated.  Now these subdivisions form a particularly well-understood class of regular matroid subdivisions, called multi-splits.  See \cite{SchroeterMultisplits}, as well as \cite{SchroeterThesis} and references therein for an excellent overview.
	
	It follows immediately from Theorem \ref{thm: weakly separated matroid subdivision blade} that their common refinement is also a regular subdivision into positroid polytopes, which in turn characterizes (the relative interior of) a cone of positive tropical Pl\"{u}cker vectors in $\text{Trop}^+G(k,n)$.
\end{proof}

\begin{conjecture}\label{conjecture: main result intro A}
	Fix an ordered set partition $\mathbf{S} = (S_1,\ldots, S_k)$ of $\{1,\ldots, n\}$, consisting of $k$ cyclically contiguous blocks $S_j$ which are cyclic intervals of the form $S_j = \{a,a+1,\ldots, b\}$.  Let $d$ be the number of blocks having size $\vert S_j\vert \ge 2$.   Then, there exists an order $(\mathbf{T}_1,\mathbf{T}_2,\ldots,\mathbf{T}_{(d-1)(k-1)-1},\mathbf{S})$, say, where 
	$$\{\mathbf{T}_1,\mathbf{T}_2,\ldots,\mathbf{T}_{(d-1)(k-1)-1}\} = \mathcal{N}_{\mathbf{S}},$$
	such that the iterated residue is nonzero and satisfies
	$$\text{Res}_{\eta_{(\mathbf{T}_{(d-1)(k-1)-1})}=0}\left(\cdots\left(\text{Res}_{\eta_{(\mathbf{T}_{1})}=0}\left(\text{Res}_{\eta_{(\mathbf{S})}=0}\left(m^{(k)}_n\right)\right)\right)\right) = \prod_{\ell=1}^d m^{(k)}_{n_\ell},$$
	after a suitable identification of kinematic parameters.  Here $n_1+\cdots +n_d = n+d(k-1)$.
\end{conjecture}
We remark that, in the conjecture, naturally one would like to have an explicit identification of indices of the form
$$m^{(k)}_{n_\ell} = m^{(k)}(S_\ell,i_1,\ldots, \hat{i_\ell},\ldots, i_k)$$
for each $\ell$.

Note that the order can be a subtle issue, because the full set of propagators involved is highly linearly dependent in general.  If the wrong order is taken the final residue may be zero!

We have already seen in Equations \eqref{eq: 18 propagators 48 A} and \eqref{eq: 18 propagators A}, that more than one factorization channel can be assigned to a given cyclic ordered set partition; we propose an extension of Conjecture \ref{conjecture: main result intro A} which reflects this.

Let $\mathbf{S} = (S_1,\ldots, S_k)$ be an ordered set partition.  Let $\widehat{\mathbf{X}}_\mathbf{S}$ be the set of all decorated ordered set partitions $(\mathbf{T},\mathbf{r}) \in \text{OSP}(\Delta_{k,n})$ which arise from $(\mathbf{S},(1,\ldots, 1))$ by lumping together adjacent blocks in all possible ways, summing the corresponding decorations of the blocks.  For example, for $\mathbf{S} = (S_1,S_2,S_3,S_4)$, then 
\begin{eqnarray*}
	& & \widehat{\mathbf{X}}_{(S_1,S_2,S_3,S_4)},\\
	& = & \{((S_1)_1 (S_2)_1 (S_3)_1 (S_4)_1),\\
	& & ((S_{12})_2 (S_3)_1 (S_4)_1), ((S_1)_1 (S_{23})_2 (S_4)_1),((S_1)_1 (S_2)_1 (S_{34})_2),((S_{14})_2 (S_2)_1 (S_3)_1),\\
	& &  ((S_{123})_3 (S_4)_1),((S_1)_1 (S_{234})_3),((S_{124})_3 (S_3)_1),((S_{134})_3 (S_2)_1),((S_{12})_2 (S_{34})_2),((S_{14})_2 (S_{23})_2)\},
\end{eqnarray*}
where we abbreviate $S_{i\cdots j} = S_i\cup \cdots \cup S_j$.

From these 11 decorated ordered set partitions we construct 11 positive tropical Pl\"{u}cker vectors, all of which are rays; but by taking linear combinations with negative integer coefficients and intersecting with $\text{Trop}^+G(4,n)$ we obtain a total of  29 rays.

However, these do not all lie in the same cone in the positive tropical Grassmannian $\text{Trop}^+G(4,n)$.  In fact they generate exactly four maximal cones, coming in two combinatorially inequivalent permutations classes, as in Equations \eqref{eq: 18 propagators 48 A} and \eqref{eq: 18 propagators A}.

Finally, we illustrate what can happen when some blocks are singletons.
\begin{example}
	For the ordered set partition $\mathbf{S} = (S_1,S_2,S_3) = (1,23,456)$, we obtain four decorated ordered set partitions
	\begin{eqnarray}
		(1_1 23_1 456_1),\ (123_2 456_1),\ (4561_2 23_1),\ (1_1 23456_2)
	\end{eqnarray}
	but among the corresponding kinematic blades only two are distinct and nonzero:
	\begin{eqnarray*}
		(\eta_{(1_1 23_1 456_1)},\ \eta_{(123_2 456_1)},\ \eta_{(4561_2 23_1)},\ \eta_{(1_1 23456_2)}) & =& (\eta_{(4561_2 23_1)},\eta_{(123_2 456_1)},\eta_{(4561_2 23_1)} ,0) \\
		& = & (\eta_{136},\eta_{236},\eta_{136}, 0),
	\end{eqnarray*}
	so that $\mathcal{N}_{\mathbf{S}} $ indexes
	$$\{ \eta_{(123_2 456_1)},\ \eta_{(4561_2 23_1)}\}.$$
	On the other hand, moving to $k=4$, if $\mathbf{S} = (1,23,45,678)$ then the nonzero kinematic blades, indexed by $\mathcal{N}_{\mathbf{S}}$ are 
	\begin{eqnarray*}
	\left\{\eta _{(123_2,45_1,678_1)},\eta _{(6781_2,23_1,45_1)},\eta _{(67812_2,2345_2)},\eta _{(12345_3,678_1)},\eta _{(678123_3,45_1)},\eta _{(456781_3,23_1)}\right\}.
\end{eqnarray*}
	If $\mathbf{S} = (1,23,4,5678)$ then $\mathcal{N}_{\mathbf{S}} $ indexes
	\begin{eqnarray*}
		\left\{\eta _{(1234_3,5678_1)},\eta _{(56781_2,234_2)},\eta _{(456781_3,23_1)}\right\}.
	\end{eqnarray*}
\end{example}

We now make a small digression which provides background needed in order to state Item (3) in Conjecture \ref{conj: general subdivision}.  In this work, we did not emphasize it, but the construction which follows plays an essential role throughout in our work with the positive tropical Grassmannian.

For each $J \in \binom{\lbrack n\rbrack}{k}^{nf}$ we define a generalized positive root by the equation 
$$v_J = \sum_{i=1}^{k-1} e_{i,\lbrack j_i-(i-1),j_{i+1}-i-1\rbrack}.$$
Of course, taking the linear dual we obtain $\gamma_J(\alpha)$ from above.

In \cite{SpeyerWilliams2003}, Speyer and Williams constructed a certain bijection $\text{Trop}\Psi:  \text{Trop}^+G(k,n) \rightarrow\mathbb{R}^{(k-1)\times (n-k)}$, which in turn induces a bijection between the quotient spaces $\text{Trop}^+G(k,n) \slash \text{Lin}_{k,n}$ and $\mathbb{T}^{(k-1,n-k)}$; now in the parameterization $M$ from Appendix \ref{sec:positive parametrization}, this latter bijection is induced by the map $\text{proj}^{Rt}_{k,n}: \mathbb{R}^{\binom{n}{k}} \rightarrow \mathbb{T}^{k-1,n-k}$, defined by 
$$\mathfrak{h}_J \mapsto \begin{cases}
	v_J, & J \in \binom{\lbrack n\rbrack}{k}^{nf}\\
	0, & J \in \binom{\lbrack n\rbrack}{k}^{frzn}
\end{cases}.$$


In the Conjecture which follows, we again allow some blocks to be singletons.

\begin{conjecture}\label{conj: general subdivision}
	With $\widehat{\mathbf{X}}_\mathbf{S}$ as above, let $\mathbf{C}$ be any maximal cone in 
	\begin{eqnarray}
		\text{span}\left\{\mathfrak{h}_{(\mathbf{S},\mathbf{r})}: (\mathbf{S},\mathbf{r}) \in \widehat{\mathbf{X}}_\mathbf{S} \right\} \cap \text{Trop}^+G(k,n).
	\end{eqnarray}
	\begin{enumerate}
		\item The cone $\mathbf{C}$ has dimension $(d-1)(k-1)$, where $d$ is the number of non-singleton blocks in $\mathbf{S}$ (we emphasize here that the cone $\mathbf{C}$ will not in general be simplicial!).
		\item The set of cones in $\text{Trop}^+G(k,n)$ containing $\mathbf{C}$ forms a polyhedral complex that is isomorphic to the Cartesian product $\text{Trop}^+G(k,n_1)\times \cdots \times \text{Trop}^+G(k,n_d)$, after modding out by the subspace spanned by $\mathbf{C}$.
		\item Let $\pi$ be in the relative interior of the cone $\mathbf{C}$.  Then, if all $c_{i,j}>0$, the face of $\mathbf{N}_{k,n}$ that minimizes the linear function, dual to the vector $\text{proj}^{Rt}_{k,n}(\pi)$,
		is combinatorially isomorphic to the Cartesian product $\mathbf{N}_{k,n_1}\times \cdots \times \mathbf{N}_{k,n_d}$.				
		\item The statement analogous to Conjecture \ref{conjecture: main result intro A} holds almost verbatim: given any collection of $(d-1)(k-1)$ compatible and linearly independent propagators which are dual to rays of $\mathbf{C}$, then the multi-dimensional residue of $m^{(k)}_n$ is either zero, or a product 
		$$\text{Res}_{\eta_{(\mathbf{T}_{(d-1)(k-1)-1})}=0}\left(\cdots\left(\text{Res}_{\eta_{(\mathbf{T}_{1})}=0}\left(\text{Res}_{\eta_{(\mathbf{S})}=0}\left(m^{(k)}(12\ldots n,12\cdots n)\right)\right)\right)\right) = \prod_{\ell=1}^d m^{(k)}_{n_\ell},$$
		depending on the order in the sequence of residues.  Here as usual $n_1+\cdots +n_d = d(k-1)$.
	\end{enumerate}
	
\end{conjecture}

\section{Wrap-up: Splitting $m^{(3)}_n$ into Biadjoint Scalars $m^{(2)}_n$; Coarsest Matroidal Subdivisions}\label{sec: wrap up}
So far we have focused on factorization of residues of $m^{(k)}_n$ as products of the form $m^{(k)}_{n_1}\cdots m^{(k)}_{n_d}$, but this is not the end of the story.  In this section, we illustrate in an example a general formula which factorizes residues of $m^{(3)}_n$ into products of the form $m^{(2)}_{n_1}\cdot m^{(2)}_{n_2} \cdot m^{(2)}_{n_3}$, where $n_1+n_2+n_3 = n+6$.  We shall start with an example and then propose a general formula.

In Figure \ref{fig:3split315 A}, we depict the dual graph to the subdivision of $\Delta_{3,15}$ that is induced by the positive tropical Pl\"{u}cker vector 
\begin{eqnarray}\label{eq:315 trop Pluck A}
	\mathfrak{h}_{2,9,15}+\mathfrak{h}_{3,9,15}+\mathfrak{h}_{4,6,15}+\mathfrak{h}_{4,7,15}+\mathfrak{h}_{4,8,15}+\mathfrak{h}_{4,9,11}+\mathfrak{h}_{4,9,12}+\mathfrak{h}_{4,9,13}+\mathfrak{h}_{4,9,14}+\mathfrak{h}_{4,9,15},
\end{eqnarray}
noting that the triples form a weakly separated collection, which means that they lie in the same cone of $\text{Trop}^+G(3,15)$, and the symmetry around the center $\mathfrak{h}_{4,9,15}$.

Here is a natural generalization of Equation \eqref{eq:315 trop Pluck A}.  For any totally nonfrozen triple $\{i,j,k\}$, define a (positive tropical Pl\"{u}cker) vector
\begin{eqnarray}
	\pi & = & \mathfrak{h}_{i,j,k} + \sum_{t = k+2}^{i-2} \mathfrak{h}_{t,j,k} + \sum_{t = i+2}^{j-2} \mathfrak{h}_{i,t,k}+\sum_{t = j+2}^{k-2} \mathfrak{h}_{i,j,t}.
\end{eqnarray}
The question is to prove that the residue of $m^{(3)}_n$ determined by $\pi$ factorizes as 
$$m^{(3)}_n \sim m^{(2)}_{j-i+2}\cdot m^{(2)}_{k-j+2}\cdot m^{(2)}_{i-k+n+2}.$$

This appears to be one of several possible ways to get such 3-splits into biadjoint scalar amplitudes $m^{(2)}_{n'}$'s; we leave the full exploration to future work.

\subsection{Proposal to Construct Coarsest Subdivisions}
In this section, we present a formula involving the noncrossing complex $\mathbf{NC}_{k,n}$ to construct a whole slew of coarsest matroid subdivisions of $\Delta_{k,n}$, and thereby lay the foundation for future work \cite{EarlyPrep}.  We formulate our main conjecture and then illustrate it with several calculations and numerous figures with explanations in the captions.

\begin{figure}[h!]
	\centering
	\includegraphics[width=0.6\linewidth]{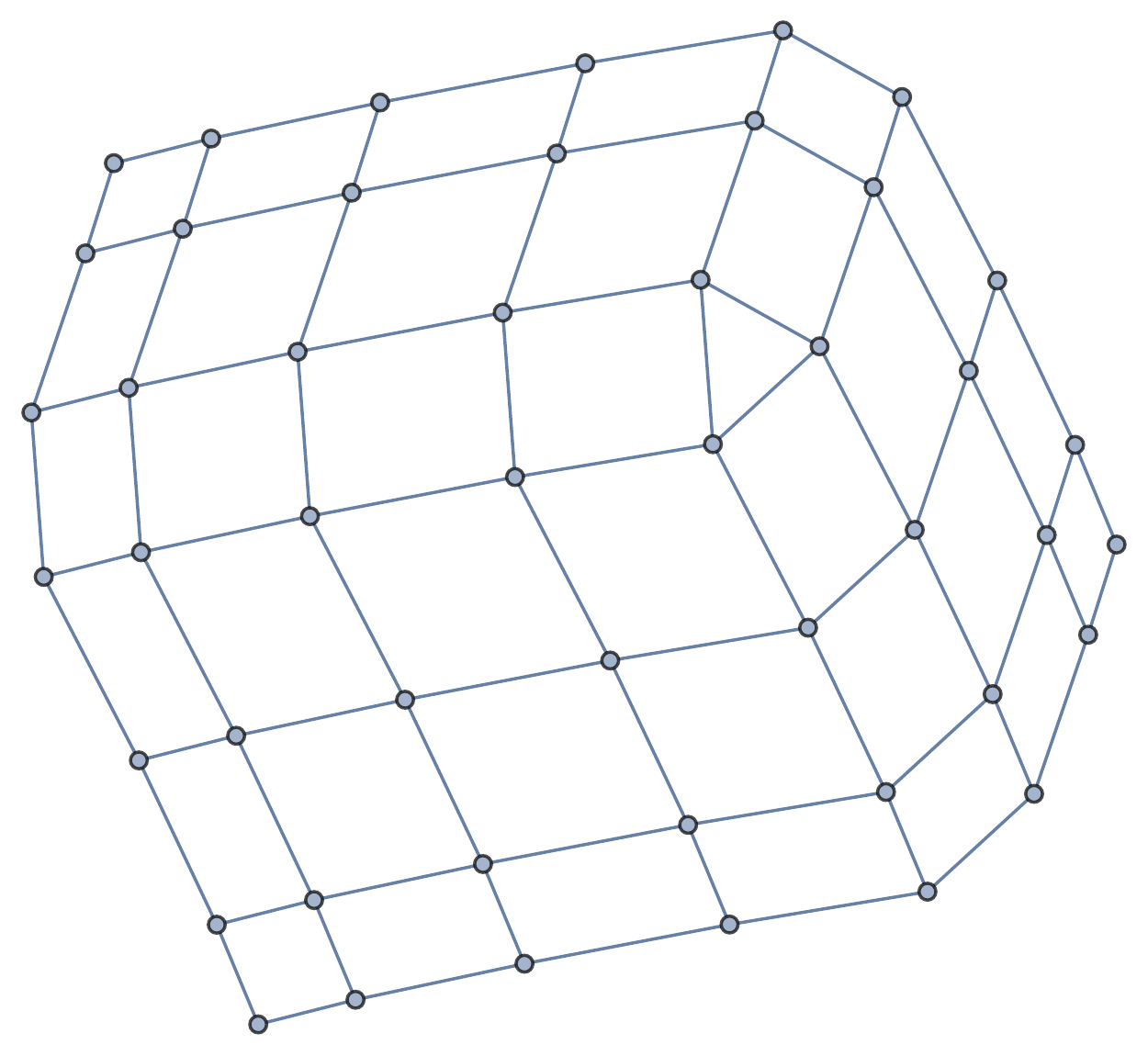}
	\caption{Dual to the positroidal subdivision of the hypersimplex $\Delta_{3,15}$ that is induced by the positive tropical Pl\"{u}cker vector in Equation \eqref{eq:315 trop Pluck A}; nodes in the graph are placed at centers of cells in the subdivision.  Edges are drawn in the graph to connect adjacent cells.  The residue of $m^{(3)}_{15}$ factorizes as a product $m^{(2)}_6 \cdot m^{(2)}_7\cdot m^{(2)}_8$.}
	\label{fig:3split315 A}
\end{figure}

\begin{conjecture}\label{conj: noncrossing expansion coarsest subdivision}
	Fix a noncrossing collection $\mathcal{J} = \{J_1,\ldots, J_m\} \in \mathbf{NC}_{k,n}$, with $m\ge 2$.  Draw a graph $\mathcal{G}$  with vertex set labeled by the collection $\mathcal{J}$ and edges $(J_i,J_j)$ whenever $\{J_i,J_j\}$ is not weakly separated.  
	
	If $\mathcal{G}$ is the complete graph on $m$ nodes, then the positive tropical Pl\"{u}cker vector
			$$\mathcal{F}^{(k)}_n\left(\sum_{j=1}^m v_{J_j}\right)$$	
			generates a ray of $\text{Trop}^+G(k,n)$; equivalently, the positroidal subdivision of $\Delta_{k,n}$ that it induces is coarsest.  This constitutes an all (k,n) prescription for constructing simple poles of $m^{(k)}_n$.
\end{conjecture}

\begin{figure}[h!]
	\centering
	\includegraphics[width=0.45\linewidth]{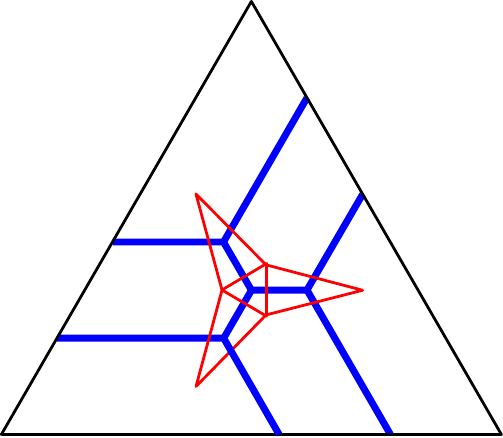}
	\caption{Matroidal weighted blade arrangement (thick, blue lines) and its dual (thin, red lines); this is not a multi-split (for multi-splits of hypersimplices, see \cite{SchroeterMultisplits}), but still any nonempty intersection of $d$ cells has the expected codimension $d-1$.  This captures a section of a certain coarsest positroidal subdivision of $\Delta_{3,12}$; it has six maximal cells and is induced by the positive tropical Pl\"{u}cker vector $\mathcal{F}^{(3)}_{12}(v_{1,6,9} + v_{2,5,10}) =-\mathfrak{h}_{1,5,9}+\mathfrak{h}_{1,5,10}+\mathfrak{h}_{1,6,9}+\mathfrak{h}_{2,5,9}$.}
	\label{fig:39sixchambersubdivision}
\end{figure}
\begin{figure}[h!]
	\centering
	\includegraphics[width=0.5\linewidth]{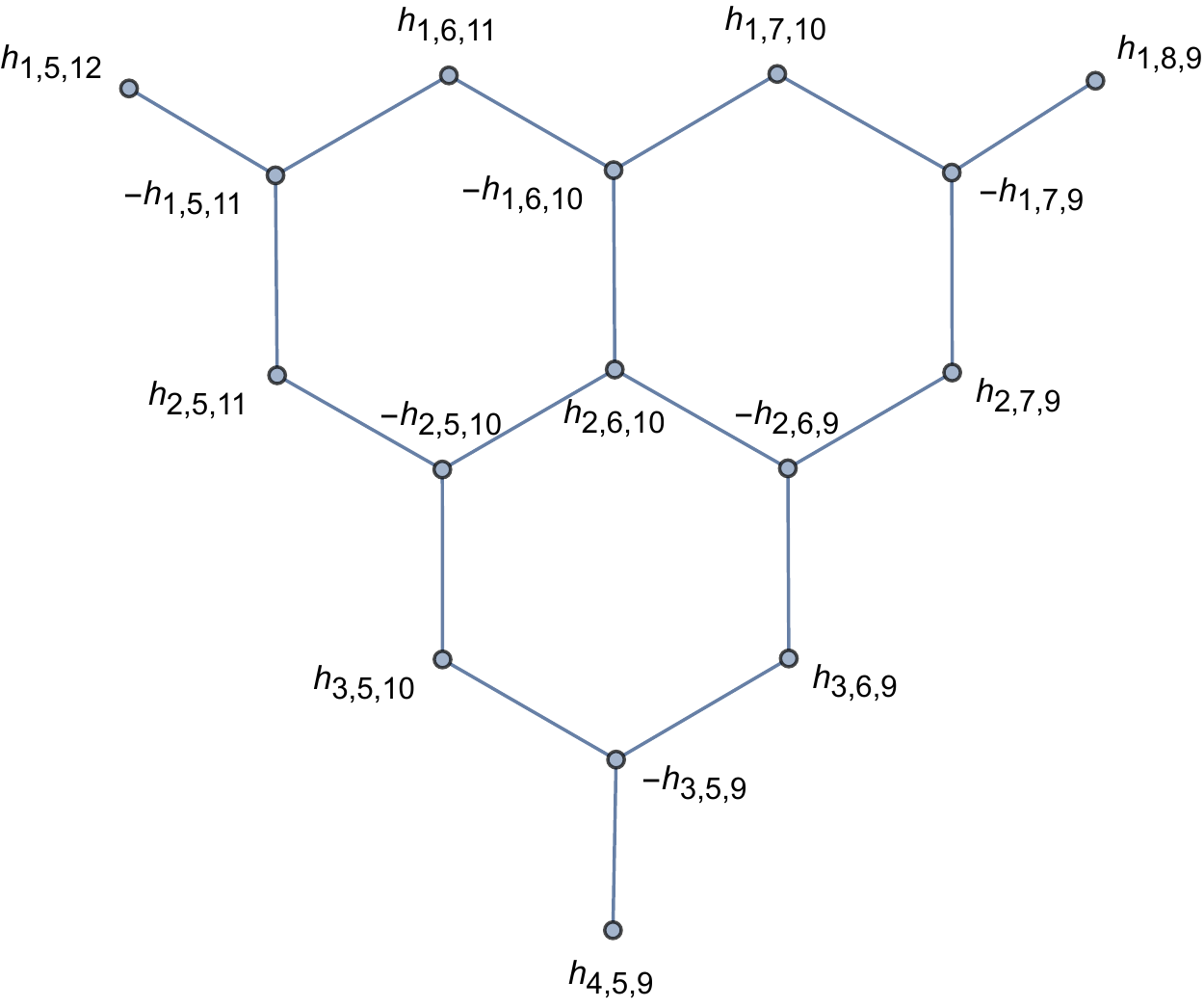}
	\caption{First nontrivial example of Conjecture \ref{conj: noncrossing expansion coarsest subdivision}.  Construction of a positive tropical Pl\"{u}cker vector from a linear combination of generalized positive roots: $\mathcal{F}^{(3)}_{12}\left(v_{1,8,9}+v_{2,7,10}+v_{3,6,11}+v_{4,5,12}\right)$.  Compare to Figure \ref{fig:missinghexagonfillededgedirections}.}
	\label{fig:completegraph312}
\end{figure}
\begin{figure}[h!]
	\centering
	\includegraphics[width=0.5\linewidth]{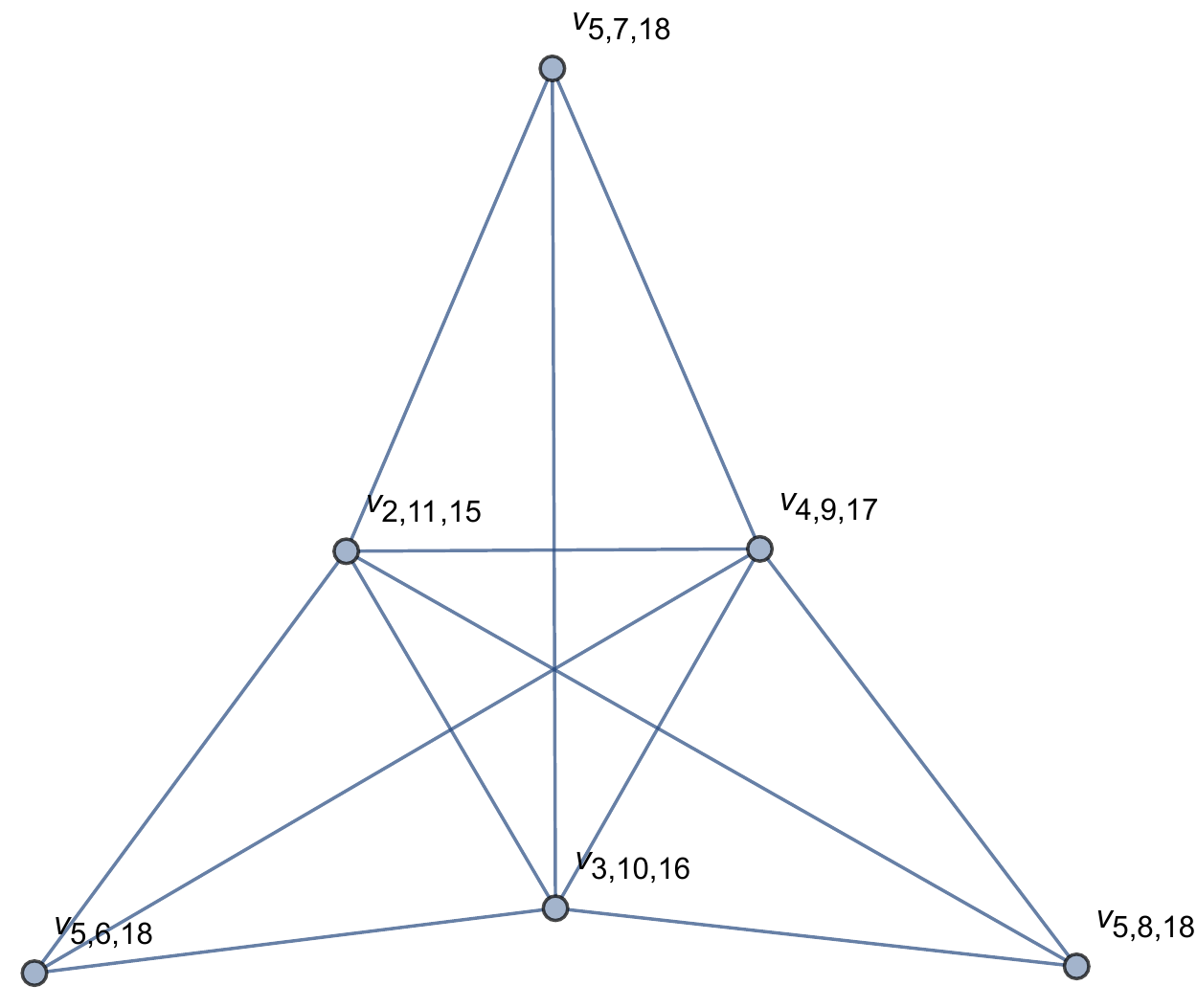}
	\caption{Noncrossing vertex set of generalized positive roots $v_{a,b,c}$; an edge connects two nodes $v_{a,b,c},v_{a',b',c'}$ if their index sets are \textit{not} weakly separated.  See Figure \ref{fig:notweaklyseparatedcoarsest318a} for the matroidal weighted blade arrangement.}
	\label{fig:notweaklyseparatedcoarsest318}
\end{figure}

\begin{figure}[h!]
	\centering
	\includegraphics[width=0.85\linewidth]{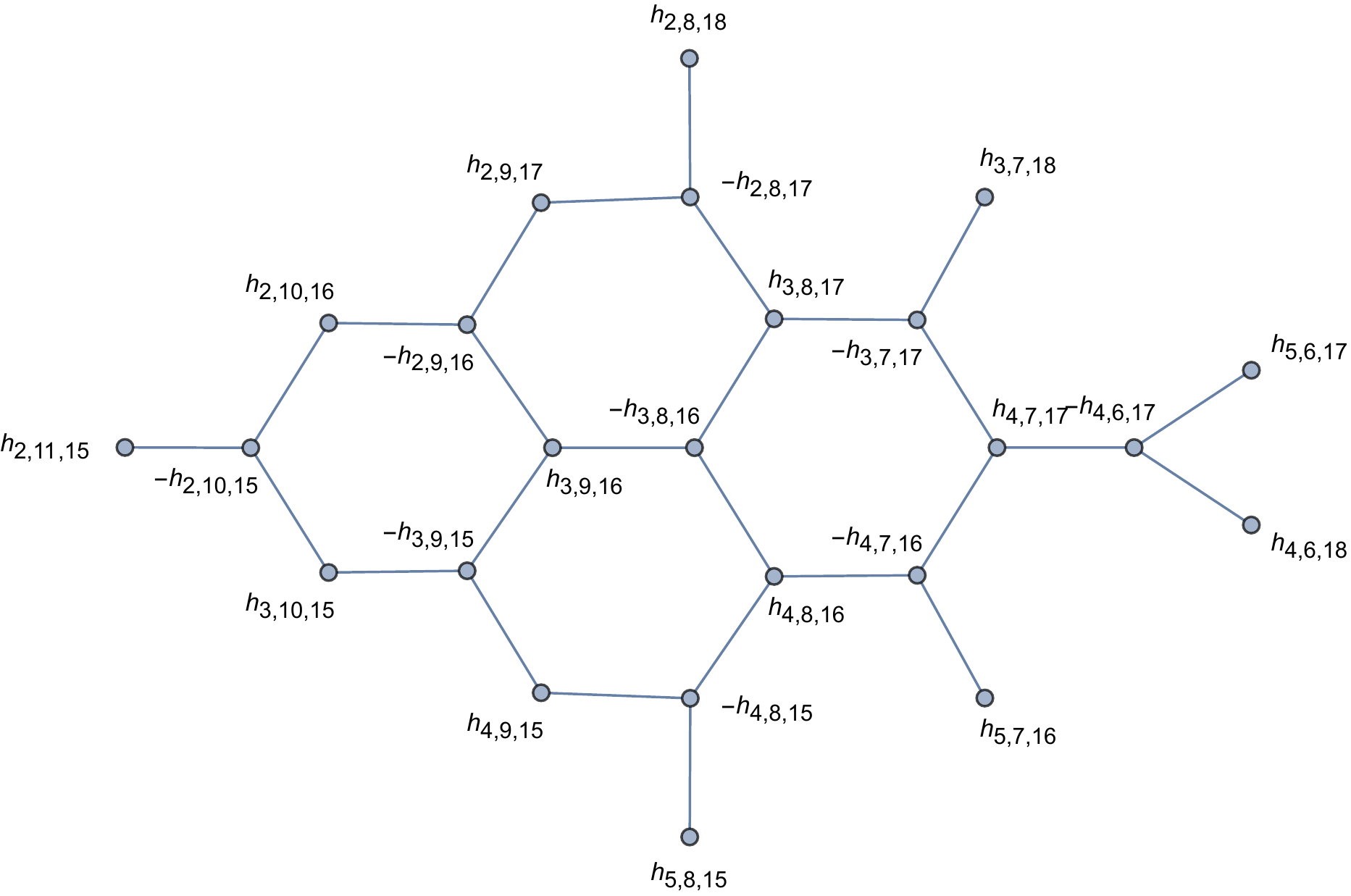}
	\caption{Locations of the blades in the matroidal weighted blade arrangement on the 1-skeleton of $\Delta_{3,19}$; take the weighted sum of the nodes to get a positive tropical Pl\"{u}cker vector which induces a coarsest positroidal subdivision of $\Delta_{3,19}$.  Replacing in this expression $\mathfrak{h}_{a,b,c}$ with $\eta_{a,b,c}$ gives a propagator for $m^{(3)}_{19}$, that is it characterizes a simple pole.  Constructed from Figure \ref{fig:notweaklyseparatedcoarsest318} as $\mathcal{F}^{(3)}_{19}(v_{2,11,15}+v_{3,10,16}+v_{4,9,17}+v_{5,6,18}+v_{5,7,18}+v_{5,8,18})$.  The graph formed from these triples $\{a,b,c\}$ does not form a complete graph, but still it gives rise to a coarsest subdivision!  This suggests the possibility to explore much further.}
	\label{fig:notweaklyseparatedcoarsest318a}
\end{figure}

\begin{example}\label{example: 48 subdivision A}
	In Figure \ref{fig:39sixchambersubdivision}, we present
	$$\mathcal{F}^{(3)}_{12}(v_{1,6,9} + v_{2,5,10}) =-\mathfrak{h}_{1,5,9}+\mathfrak{h}_{1,5,10}+\mathfrak{h}_{1,6,9}+\mathfrak{h}_{2,5,9}$$
	as a matroidal weighted blade arrangement.
	
	In Figure \ref{fig:completegraph312} we present the ray $\mathcal{F}\left(v_{1,8,9}+v_{2,7,10}+v_{3,6,11}+v_{4,5,12}\right)$ of the positive tropical Grassmannian $\text{Trop}^+G(3,12)$.  In the summation, the graph formed with nodes the indices of the generalized positive roots $v_{a,b,c}$ is a complete graph, with edges as defined in Conjecture \ref{conj: noncrossing expansion coarsest subdivision}, where pairs of nodes are connected if their labels are not weakly separated.
	
	Also consider  
	\begin{eqnarray*}
		& & \mathcal{F}^{(4)}_8(v_{1,4,6,7}+v_{2,3,6,8}+v_{2,4,5,8}) \\
		& = & \mathfrak{h}_{1,3,5,7}-\mathfrak{h}_{1,3,5,8}-\mathfrak{h}_{1,3,6,7}+\mathfrak{h}_{1,3,6,8}-\mathfrak{h}_{1,4,5,7}+\mathfrak{h}_{1,4,5,8}+\mathfrak{h}_{1,4,6,7}-\mathfrak{h}_{2,3,5,7}+\mathfrak{h}_{2,3,5,8}+\mathfrak{h}_{2,3,6,7}+\mathfrak{h}_{2,4,5,7},
	\end{eqnarray*}
	see Figure \ref{fig:48subdivisioncoarsesta}.  We include several more examples with explanations in the captions.
\end{example}
\begin{figure}[h!]
	\centering
	\includegraphics[width=0.55\linewidth]{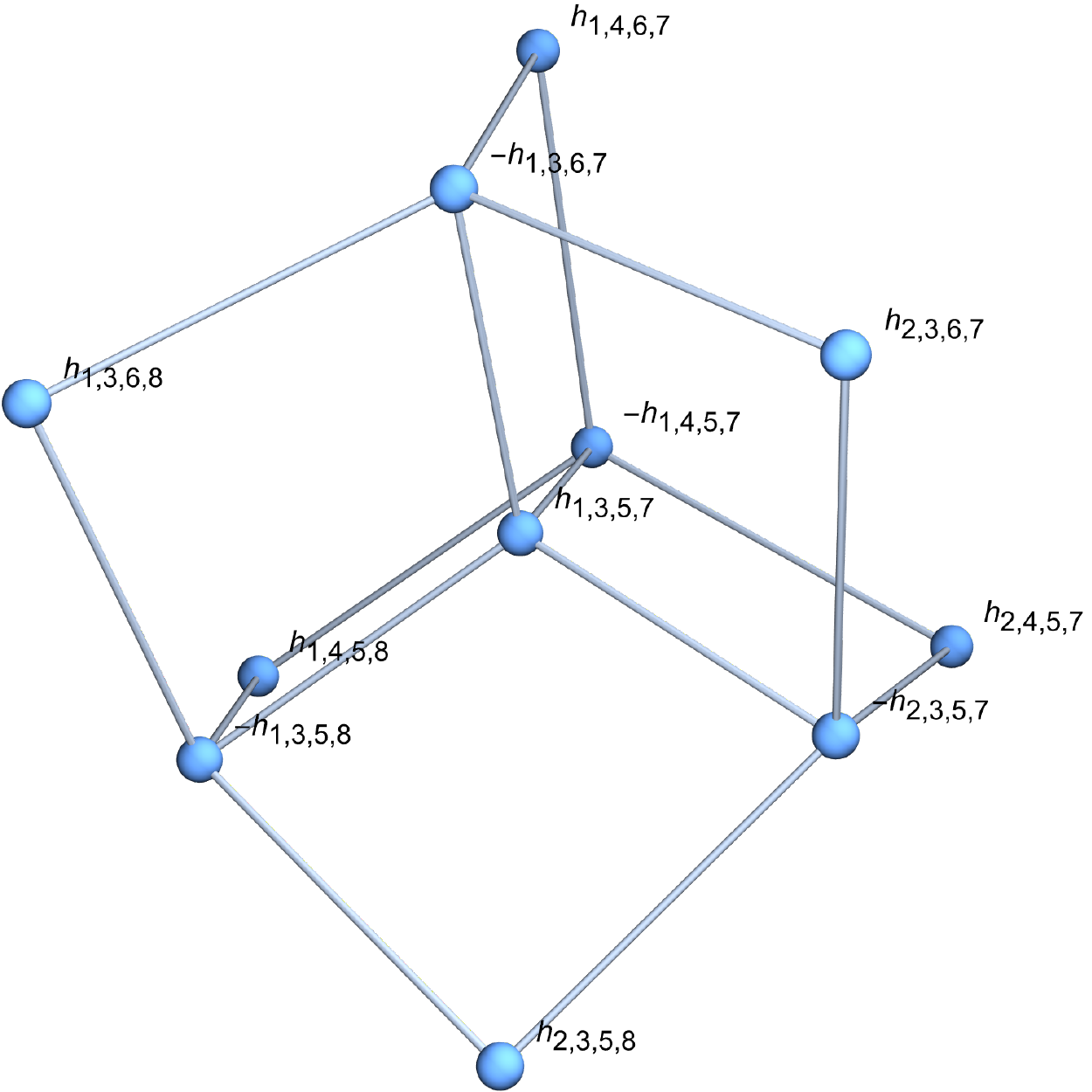}
	\caption{Locations of the blades on the 1-skeleton of the hypersimplex $\Delta_{4,8}$ for Example \ref{example: 48 subdivision A}.  The translation is that if $\mathfrak{h}_{a,b,c,d}$ appears then the blade $\beta = ((1,2,\ldots, 8)$ is placed at the vertex $e_{abcd} := e_a+e_b+e_c+e_d$ of $\Delta_{4,8}$, with the coefficient as a weight.  To be explicit, the positive tropical Pl\"{u}cker vector can be obtained by evaluating $\mathcal{F}^{(4)}_8(v_{1,4,6,7}+v_{2,3,6,8}+v_{2,4,5,8})$.}
	\label{fig:48subdivisioncoarsesta}
\end{figure}

\begin{figure}[h!]
	\centering
	\includegraphics[width=0.9\linewidth]{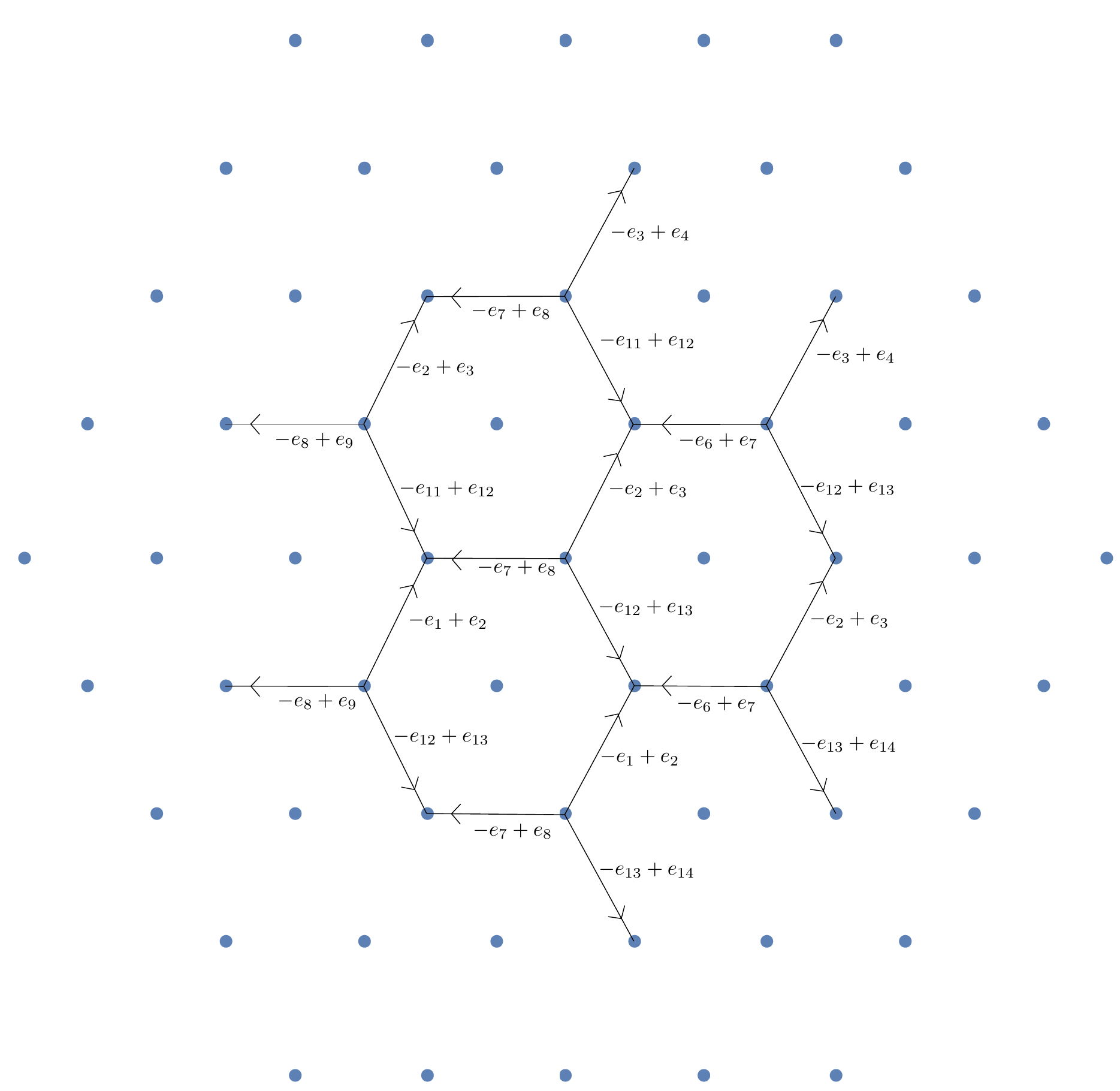}
	\caption{Blade-theoretic representation of a coarsest positroidal subdivision, a ray in $\text{Trop}^+G(3,n)$ for large $n\gg 14$, say; here it is shown embedded in the \textit{root lattice} of $SL_3$.  Labeling the edges (with translation vectors $e_i-e_{i+1}$) can be helpful, in addition to the (equivalent) direct labeling of the vertices with blades $\beta_{a,b,c}$.  Note the structure of the sets of parallel edges: labels are the same.  Non-uniform hexagonal tessellations can be obtained by replacing parallel directions $e_i-e_{i+1}$ correspondingly with larger intervals $-e_i+e_{i+r_i}$ with $r_i\ge2$.  See Figure \ref{fig:hexagonzonotopaldeformation}.  Arrows on the edges provide an intuitive way to encode $\pm1$ coefficients.}
	\label{fig:missinghexagonfillededgedirections}
\end{figure}

\begin{figure}[h!]
	\centering
	\includegraphics[width=1\linewidth]{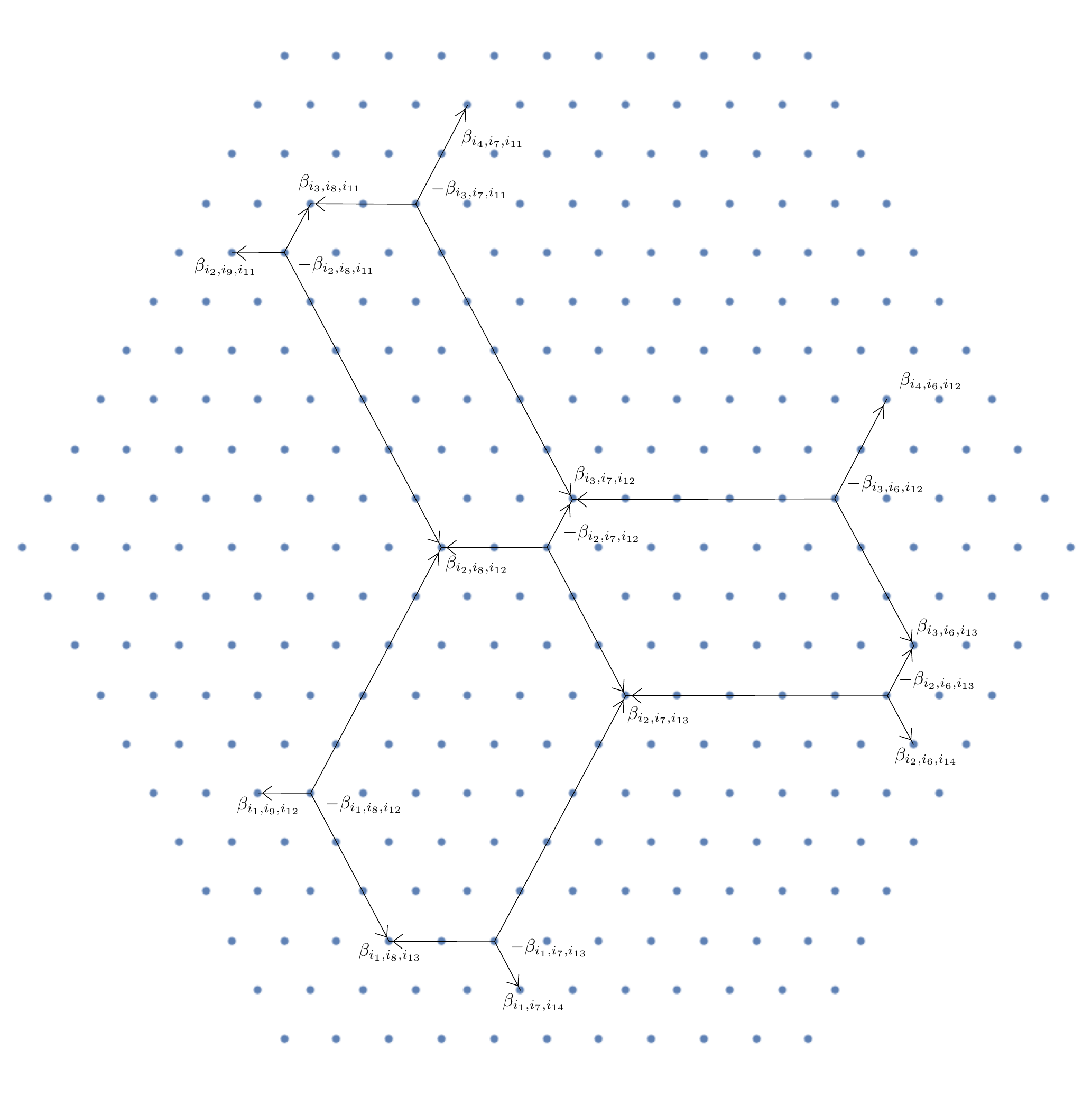}
	\caption{A possible deformation of Figure \ref{fig:missinghexagonfillededgedirections}, shown here embedded into the root lattice for $SL_3$. Each hexagon is a projection of a cube in the 1-skeleton of $\Delta_{3,n}$.  Here we have not specified the values of the elements $i_j$; rather we specify their \textit{consecutive differences}, by the lengths of the line segments: for instance, the edge connecting $-\beta_{i_2,i_7,i_{12}}$ and $\beta_{i_2,i_7,i_{13}}$ has length 3, which imposes the condition $i_{13} - i_{12} = 3$.  Here we have embedded $(1,2,\ldots, 14)$ as a subcyclic order $(i_1,i_2,\ldots, i_{14})$ of $(1,2,\ldots, n)$ for some large $n\gg 14$.}
	\label{fig:hexagonzonotopaldeformation}
\end{figure}

\section{Discussion}
In this paper, we have studied the problem of factorization for CEGM amplitudes $m^{(k)}_n$, that is, to classify the residues $m^{(k)}_n$ which are product of objects of the same type; for the first time, we have been able to peek deep into the factorization poset of the positive tropical Grassmannian $\text{Trop}^+G(k,n)$ for all (k,n).  This borders on an important problem in combinatorial geometry, where one would like to explore higher codimension faces of secondary polytopes.

One pressing question is to construct systematically the \textit{simple} poles of $m^{(k)}_n$; this is equivalent to classifying the coarsest regular subdivisions of hypersimplices $\Delta_{k,n}$.  We have summarized our first steps in this direction and will return to them in \cite{EarlyPrep}.  

A fascinating direction starts from the very basic observation that CEGM amplitudes (at least, in the special case of the standard cyclic order $\mathbb{I} = (1,2,\ldots, n)$, so that $m^{(k)}_n := m^{(k)}_n(\mathbb{I},\mathbb{I})$) and amplituhedra $\mathcal{A}_{n,d,m}$, introduced by Arkani-Hamed and Trnka in \cite{ArkaniHamedTrnka}, are constructed from the nonnegative Grassmannian, but by quite different means.  We would like to ask what lies between the two; some relations with subdivisions have been explored with the $m=2$ amplituhedron \cite{LPW2020,PSW2021}; can one go further?  

Another interesting question is to investigate how our proposal for factorization might extend towards the work of Arkani-Hamed, He and Lam in the context of certain stringy integrals \cite{AHL2019Stringy}.  Generalized biadjoint amplitudes, roughly speaking, compute stringy integrals to leading order, so it seems natural to ask what happens when their higher order terms are taken into account.

Finally, let us point out that technology was developed in \cite{E2021} that can be used to further check (and possibly prove) our conjectures.  Indeed, in \cite{E2021}, there were two main results relevant here.  First, that \textit{generalized positive roots} 
$$\gamma_J(\alpha) = \sum_{i=1}^{k-1} \alpha_{i,\lbrack j_i-(i-1),j_{i+1}-i-1\rbrack}$$
 are (in duality with) the normal vectors to the $\binom{n}{k}-n$ facets of the PK polytope \cite{CE2020B}, and second that any positive tropical Pl\"{u}cker vector projects uniquely onto a positive linear combination of generalized roots, which in turn can be uniquely characterized by a certain weighted noncrossing collection of noncrossing $k$-tuples, in the noncrossing complex $\mathbf{NC}_{k,n}$!  These collections appear to carry a nontrivial amount of information about coarsest positroidal subdivisions in particular, as we have suggested in Conjecture \ref{conj: noncrossing expansion coarsest subdivision}.  This conjecture aims to construct quite general families of coarsest positroidal subdivisions and it should be possible to take certain degenerations in various ways to obtain more, and in particular to finally give a systematic construction of the simple poles of $m^{(k)}_n$, see \cite{EarlyPrep}.

\section*{Acknowledgements}
This work has benefited from discussions with and support from many people.  We thank Nima Arkani-Hamed, Johannes Henn, Lukas Kuhne, Leonid Monin, Matteo Parisi, Benjamin Schroeter, Bernd Sturmfels, Alexander Tumanov, Bruno Umbert, Lauren Williams, and Yong Zhang, and especially Freddy Cachazo, for stimulating discussions and helpful comments on a draft.  We also thank Jianrong Li for helpful discussions and for proofreading.

We thank the Institute for Advanced Study for excellent working conditions while this project was initiated.  This research received funding from the European Research Council (ERC) under the European Union's Horizon 2020 research and innovation programme (grant agreement No 725110), Novel structures in scattering amplitudes.

	\appendix
	
\section{Kinematic Blades}\label{sec: kinematic blades}

In what follows, we give the algorithm to translate between the combinatorial data for the two formulas for the planar basis elements.  The objective is to produce for each $J \in \binom{\lbrack n\rbrack}{k}^{nf}$ a decorated ordered set partition $(\mathbf{S},\mathbf{r})$, such that 
$$\eta_J = \eta_{(\mathbf{S},\mathbf{r})}.$$
See Proposition \ref{prop: equality planar basis}; the proof relies on the characterization of matroidal blade arrangements \cite{Early19WeakSeparationMatroidSubdivision} and a straightforward calculation with piecewise-linear functions.

Here we summarize the data structure for the formula.

\begin{itemize}
	\item Input: a $k$-element subset $J \in \binom{\lbrack n\rbrack}{k}$.
	\item Output: a decorated ordered set partition $(\mathbf{S},\mathbf{r})$  of type $\Delta_{k,n}$.
	\item Key property (Proposition \ref{prop: equality planar basis}): $\eta_J = \eta_{(\mathbf{S},\mathbf{r})}$.
\end{itemize}

\begin{defn}[\cite{Early19WeakSeparationMatroidSubdivision}]\label{def: desp bijection}
	Given\footnote{In practice we will usually assume that $J$ is nonfrozen, that is not a single cyclic interval $J = \{j,j+1,\ldots, j+k-1\}$} $J \in \binom{\lbrack n\rbrack}{k}$, we construct a decorated ordered set partition of type $\Delta_{k,n}$ by 
	$$(\mathbf{S},\mathbf{r}) = ((S_1)_{r_1},\ldots, (S_d)_{r_d}),$$
	where $(S_j,r_j)  =  (J_j\cup C_j , \vert J_j \vert ),$ and where $J_j, C_j$ are constructed as follows.
	Let
	$$J=\{i_1,i_1+1,\ldots, i_1+({\lambda_1}-1)\} \cup \{i_2,i_2+1,\ldots, i_2+({\lambda_2}-1)\}\cup \cdots\cup \{i_{d},i_{d}+1,\ldots, i_d + (\lambda_{d}-1)\},$$
	be the decomposition of $J$ into cyclic intervals, so that we have $\lambda_1+\cdots+ \lambda_d = k$.   Denote $J_j = \{i_j,i_j+1,\ldots, i_j+({\lambda_j}-1)\}$ where without loss of generality we assume that $1\in J_1$.  Let $(C_1,\ldots, C_d)$ be the interlaced complement to the intervals in $J$, so that we have the concatenation to the standard cyclic order
	$$(C_1,J_1,C_2,J_2,\ldots, C_d,J_d)=(\ldots, n,1,2,\ldots).$$
	In other words, the intervals $J_j$ are the positions of consecutive one's, and the intervals $C_j$ are the positions of consecutive zero's, in the 0/1 vector $e_J = \sum_{i\in J} e_i$.  Then 
	$$(\mathbf{S},\mathbf{r}) = ((S_1)_{r_1},(S_2)_{r_2},\ldots, (S_d)_{r_d})$$
	is the decorated ordered set partition of type $\Delta_{k,n}$ defined by 
	\begin{eqnarray}\label{eq: dOSP}
		(S_j,r_j) & = & (J_j\cup C_j , \vert J_j \vert ).
	\end{eqnarray} 
\end{defn}

Let us set some notation for the proof of the Proposition \ref{prop: equality planar basis}.  For each $J \in \binom{\lbrack n\rbrack}{k}^{nf}$, define a piecewise-linear function $\rho_J:\mathbb{R}^n \rightarrow \mathbb{R}$, by 
$$\rho_J(x) = \min\{L_1(x-e_J),\ldots, L_n(x-e_J)\},$$
where we remind that $L_1,\ldots, L_n$ are given by 
\begin{eqnarray*}
	L_j(x) & = & x_{j+1}+2x_{j+2} + \cdots +(n-1)x_{j-1}.
\end{eqnarray*}
Then we can recover our formula for $\eta_J$ by localizing $\rho_J(x)$ to the vertex set $\left\{e_I: I \in \binom{\lbrack n\rbrack}{k}\right\}$ of the hypersimplex $\Delta_{k,n}$,
$$\eta_J = -\frac{1}{n}\sum_{I \in \binom{\lbrack n\rbrack}{k}}\rho_J(e_I)\mathfrak{s}_I.$$

\begin{prop}\label{prop: equality planar basis}
	Given $J\in \binom{\lbrack n\rbrack}{k}^{nf}$, let $(\mathbf{S},\mathbf{r}) = ((S_1)_{r_1},\ldots, (S_d)_{r_d})$ be the decorated ordered set partition defined in Equation \eqref{eq: dOSP}.  Then we have the equality
	$$\eta_{J} = \eta_{(\mathbf{S},\mathbf{r})}.$$
\end{prop}

\begin{proof}
	Let $J \in \binom{\lbrack n\rbrack}{k}^{nf}$ and $(\mathbf{S},\mathbf{r})$ be as in the statement of Definition \ref{def: desp bijection}.  We will show that the two piecewise-linear functions $\rho_J(x)$ and $\rho_{(\mathbf{S},\mathbf{r})}(x)$, when restricted to the hypersimplex $\Delta_{k,n}$, are linear on exactly the same $d$ matroid polytopes; from this it will follow that they induce the same regular matroid subdivision of $\Delta_{k,n}$.  Since they both induce coarsest subdivisions, they must differ only by an element of the lineality space together with a possible over all (positive) scaling.
	
	In more detail, we need to show that the piecewise-linear functions $\rho_J\big\vert_{\Delta_{k,n}}$ and $\rho_{(\mathbf{S},\mathbf{r})}\big\vert_{\Delta_{k,n}}$ are linear on exactly the same set of $d$ cells, \textit{hypersimplicial plates} in $\Delta_{k,n}$, denoted 
	$$\lbrack (S_j)_{r_j},(S_{j+1})_{r_{j+1}},\ldots, (S_{j-1})_{r_{j-1}}\rbrack,$$ 
	that are characterized by the facet inequalities,
	\begin{eqnarray*}
		x_{S_j} & \ge & r_j\\
		x_{S_j} + x_{S_{j+1}} & \ge & r_j + r_{j+1}\\
		& \vdots & \\
		x_{S_j} + x_{S_{j+1}} + \cdots + x_{S_{j-2}} & \ge & r_j + r_{j+1} + \cdots + r_{j-2},
	\end{eqnarray*}
	for $j=1,\ldots, d$.  The union of the internal faces of the subdivision is called a blade\footnote{Blades were first defined by A. Ocneanu in \cite{OcneanuVideo}.}, and is denoted $(((S_j)_{r_j},(S_{j+1})_{r_{j+1}},\ldots, (S_{j-1})_{r_{j-1}}))$.
	
	The proof the restriction of $\rho_J$ induces a subdivision of this form follows by combining \cite[Proposition 6]{Early19WeakSeparationMatroidSubdivision} and \cite[Theorem 17]{Early19WeakSeparationMatroidSubdivision}: it is a (coarsest) matroid subdivision, called a $d$-split (\cite{Herrmann,SchroeterMultisplits}).  This implies that the two restricted functions differ by an element of the lineality space; since the subdivision is coarsest, the only other degree of freedom is a possible over all normalization.  Checking that the normalizations given match is left to the reader.  Clearly the relations coming from momentum conservation,
	$$\sum_{J\ni j}\mathfrak{s}_J = 0,\ j=1,\ldots, n$$
	are in duality with the generators of the lineality subspace 
	$$\text{span}\left\{\sum_{J\ni j}e^J: j=1,\ldots, n\right\}$$
	of $\mathbb{R}^{\binom{n}{k}}$, where the elements $e^J$ for $J \in \binom{\lbrack n\rbrack}{k}$ are the standard basis.
	
	It follows that (modulo momentum conservation)
	$$\sum_{I\in \binom{\lbrack n\rbrack}{k}}\rho_J(e_I)\mathfrak{s}_I = \sum_{I\in \binom{\lbrack n\rbrack}{k}}\rho_{(\mathbf{S},\mathbf{r})}(e_I)\mathfrak{s}_I,$$
	hence
	$$\eta_J = \eta_{(\mathbf{S},\mathbf{r})}.$$
\end{proof}

\begin{example}
	It straightforward to check the following identities; here the right column indicates the corresponding generalized biadjoint scalar.
	\begin{eqnarray*}
		\eta_{(12_1 3456_1)}  = & \eta_{26}:&\ \ m^{(2)}_6,\\
		\eta_{(123_1 456_2)}  =  &\eta_{356}:&\ \ m^{(3)}_6,\\
		\eta_{(712_1 34_1 56_1)}  = & \eta_{246}:&\ \ m^{(3)}_7,\\
		\eta_{(12_1 345_1 6789_2)}  =  &\eta_{2589}:&\ \ m^{(4)}_9,\\
		\eta_{(12_1 34_1 567_1 89_1)}   =& \eta_{2479}:& \ \ m^{(4)}_9.
	\end{eqnarray*}
	Then for instance one has the following expression for the planar kinematic invariant $s_{12} = s_{345}$ which gives rise to a nonzero residue for $m^{(2)}_5$, respectively
	\begin{eqnarray*}
	\eta_{(12_1 345_1)} &  = & -\mathfrak{s}_{12}-\mathfrak{s}_{34}-\mathfrak{s}_{35}-\mathfrak{s}_{45}-\frac{3 \mathfrak{s}_{13}}{2}-\frac{3 \mathfrak{s}_{14}}{2}-\frac{3 \mathfrak{s}_{15}}{2}-\frac{3 \mathfrak{s}_{23}}{2}-\frac{3 \mathfrak{s}_{24}}{2}-\frac{3 \mathfrak{s}_{25}}{2}\\
	& = & \eta_{25},
\end{eqnarray*}
	and for example
	\begin{eqnarray*}
	\eta_{(14_1 26_1 35_1)} & = & \frac{1}{3}\left(-6 \mathfrak{s}_{123}-4 \mathfrak{s}_{124}-6 \mathfrak{s}_{125}-5 \mathfrak{s}_{126}-5 \mathfrak{s}_{134}-4 \mathfrak{s}_{135}-6 \mathfrak{s}_{136}-5 \mathfrak{s}_{145}-4 \mathfrak{s}_{146}-6 \mathfrak{s}_{156}\right.\\
& - & \left.6 \mathfrak{s}_{234}-5 \mathfrak{s}_{235}-4 \mathfrak{s}_{236}-6 \mathfrak{s}_{245}-5 \mathfrak{s}_{246}-4 \mathfrak{s}_{256}-4 \mathfrak{s}_{345}-6 \mathfrak{s}_{346}-5 \mathfrak{s}_{356}-6 \mathfrak{s}_{456}\right)\\
& = &  -\eta _{124}+\eta _{125}-\eta _{135}+\eta _{136}-\eta _{146}-\eta _{236}+\eta _{245}+\eta _{246}-\eta _{256}.
	\end{eqnarray*}
\end{example}

\section{Positive Parametrization}\label{sec:positive parametrization}
In this Appendix we construct an embedding 
$$(\mathbb{CP}^{n-k-1})^{\times (k-1)} \hookrightarrow X(k,n)$$
of a Cartesian product of projective spaces into $X(k,n)$.

We first define a $(k-1)\times (n-k-1)$ polynomial-valued matrix $M_{k,n}$ with entries $m_{i,j}(x)$, with $(i,j) \in \lbrack 1,k-1\rbrack \times \lbrack 1,n-k\rbrack$, defined by 
\begin{eqnarray*}
	m_{i,j}(\{x_{a,b}: (a,b) \in \lbrack i,k-1\rbrack \times \lbrack 1,j\rbrack\}) & = & \sum_{1\le b_i\le b_{i+1}\le \cdots \le b_{k-1}\le j}\left(x_{i,b_{i}}x_{i+1,b_{i+1}}\cdots x_{k-1,b_{k-1}}\right).
\end{eqnarray*}
For instance, fixing $(k,n) = (4,n)$ then 
\begin{eqnarray*}
	m_{1,2} & = & x_{1,1} x_{2,1} x_{3,1}+x_{1,1} x_{2,1} x_{3,2}+x_{1,1} x_{2,2} x_{3,2}+x_{1,2} x_{2,2} x_{3,2}\\
	m_{2,3} & = & x_{2,1} x_{3,1}+x_{2,1} x_{3,2}+x_{2,2} x_{3,2}+x_{2,1} x_{3,3}+x_{2,2} x_{3,3}+x_{2,3} x_{3,3}\\
	m_{3,4} & = & x_{3,1}+x_{3,2}+x_{3,3}+x_{3,4}.
\end{eqnarray*}

For the embedding $(\mathbb{CP}^{n-k-1})^{\times k-1} \hookrightarrow X(k,n)$, we construct a $(k-1)\times (n-k)$ matrix with $M_{k,n}$ as its upper right block:
\begin{eqnarray}\label{eq: pos parametrization}
	M & = & \begin{bmatrix}
		1 &  &  & 0 & m_{1,1}& \cdots  & m_{1,n-k} \\
		& \ddots &  &  &\vdots & \ddots & \vdots \\
		&  & 1 &  & m_{k-1,1} & & m_{k-1,n-k} \\
		0&  &  & 1 & 1 & \cdots & 1 \\
	\end{bmatrix}.
\end{eqnarray}
For instance, for rank $k=3$ we have
\begin{eqnarray}
	M_{3,6} & = & \begin{bmatrix}
		m_{1,1} & m_{1,2} & m_{1,3}\\
		m_{2,1} & m_{2,2} & m_{2,3}
	\end{bmatrix}\\	
	& = & \begin{bmatrix}
		x_{1,1}x_{2,1} & x_{1,1}x_{2,12}+x_{1,2}x_{2,2}  & x_{1,1}x_{2,123} + x_{1,2}x_{2,23} + x_{1,3}x_{2,3} \\
		x_{2,1} & x_{2,12} & x_{2,123}\nonumber
	\end{bmatrix}
\end{eqnarray}
and for the embedding we have
$$M = \begin{bmatrix}
	1 & 0 & 0 & x_{1,1}x_{2,1} & x_{1,1}x_{2,12}+x_{1,2}x_{2,2} & x_{1,1}x_{2,123} + x_{1,2}x_{2,23} + x_{1,3}x_{2,3} \\
	0& 1 & 0 & x_{2,1} & x_{2,12} & x_{2,123}\\
	0 & 0 & 1 & 1 & 1 & 1 
	\nonumber
\end{bmatrix}.$$

\section{CEGM Amplitudes}

Fixing a projective frame on $\mathbb{CP}^{k-1}$, say the standard one
$$(g_1,\ldots, g_{k+1})=(e_1,e_2,\ldots, e_k,e_1+\cdots+e_k),$$
the scattering equations are defined with respect to inhomogeneous coordinates
$$\{z_{i,j}:i=1,\ldots, k-1,\ j = k+2,\ldots, n\}.$$
The affine chart then becomes
\begin{eqnarray}\label{eq:affine chart}
	\mathcal{C}=\left\{\begin{bmatrix}
		1 & 0 & \cdots  & 0 & 0 & 1 & z_{1,k+2} &   & z_{1n} \\
		0 & 1 &  &  & \vdots  & 1 & z_{2,k+2} & & z_{2n}\\
		\vdots  &  & \ddots &  &  & \vdots  & \vdots &\cdots  & \\
		0 & \cdots  &  & 1 & 0  & 1 & z_{k-1,k+2} &  & z_{k-1,n}\\
		0 & 0 & \cdots & 0 & 1 & 1 & 1 &  & 1
	\end{bmatrix}: z_{ij}\in\mathbb{C}\right\}.
\end{eqnarray}
\begin{defn}
	With respect to the given chart $\mathcal{C}$, the scattering equations are the $(k-1)(n-k-1)$ equations
	$$\left\{\frac{\partial \mathcal{P}_{k,n}}{\partial z_{ij}}: 1\le i \le k-1,\ k+2\le j\le n \right\},$$
	where $z_{ij}$ are the inhomogeneous coordinates on the chart $\mathcal{C}$.
\end{defn}
\begin{rem}
	Of course, one can show that solutions to the scattering equations are actually independent of the chart $\mathcal{C}$.  
\end{rem}

\subsection{Jacobian Matrix}

In the computation of the generalized biadjoint scalar $m^{(k)}(\alpha,\alpha)$, one finds the so-called reduced Jacobian determinant $\det'\Phi^{(k)}$ of the gradient of the potential function $\mathcal{P}_{k,n}$.  One often (but not always) fixes an affine chart as in Equation \ref{eq:affine chart}.  We choose not to; the same procedure still works, but now there is one subtle issue: the Hessian matrix how has a nontrivial kernel.  There is a solution which introduces a ``$\det'$'' Hesssian determinant, which we now describe.

First define the Hessian matrix 
$$\Phi^{(k)} = \left(\frac{\partial^2 \mathcal{P}_{k,n}}{\partial w_{ab}\partial w_{cd}}\right)_{(a,b),(c,d)\in \{1,\ldots, k-1\}\times \{1,\ldots, n\}}.$$
Here $(w_{1j},w_{2j},\ldots, w_{k-1,j},1)$ is an affine chart on the $j^{th}$ copy of $\mathbb{CP}^{k-1}$.

Denote by 
$$V_{i_1\cdots i_{k+1}} := \prod_{j=1}^{k+1}\det(g_{i_1},\ldots \widehat{g_{i_j}}\cdots g_{i_{k+1}}),$$
which is a natural $(k+1)\times(k+1)$ generalization of the usual $3\times 3$ Vandermonde determinant in the case $k=2$,  here for any $(k+1)$-element subset $\{i_1,\ldots, i_{k+1}\} \in \binom{\lbrack n\rbrack}{k+1}$ of $\{1,\ldots, n\}$.

Following \cite{CEGM2019}, choose any pair of $(k+1)$-element subsets $A, B$ of the columns $\{1,\ldots, n\}$, let 
$$\lbrack \Phi^{(k)}\rbrack^A_{B} = \left(\frac{\partial^2 \mathcal{P}_{k,n}}{\partial w_{a_1a_2}\partial w_{b_1b_2}}\right)^{(a_1,a_2)\in\lbrack k-1\rbrack \times(\lbrack n\rbrack \setminus A)}_{(b_1,b_2)\in\lbrack k-1\rbrack \times(\lbrack n\rbrack \setminus B)}$$
be the \textit{reduced} Hessian matrix of $\mathcal{P}_{k,n}$, which can be obtained directly from the Hessian matrix $\Phi^{(k)}$ by deleting from it all rows labeling points in $A$ and all columns labeling points in $B$.

\begin{rem}
	The ratio 
	$$\frac{\det\left(\lbrack \Phi^{(k)}\rbrack^A_{B}\right)}{V_A V_B}$$
	is independent of the choice of subsets $A,B \subset \lbrack n\rbrack$.  
	
	Therefore one defines
	$${\det}'\left( \Phi^{(k)}\right) = \frac{\det\left(\lbrack \Phi^{(k)}\rbrack^A_{B}\right)}{V_A V_B},$$
	where it is understood that the choice of $(k+1)$-element subsets $A,B \subset  \lbrack n\rbrack$ is fixed at the outset.
\end{rem}

\subsection{Generalized Biadjoint Scalar Amplitudes and the Global Schwinger Parametrization}\label{subsec:generalizedBiadjointScalar}
For a fixed cyclic order $\alpha = (\alpha_1,\ldots, \alpha_n)$, the generalized Parke-Taylor factor is given by 
$$PT^{(k)}\left(\alpha\right) = \frac{1}{p_{\alpha_1\alpha_2\cdots \alpha_k}p_{\alpha_2\alpha_3\cdots \alpha_{(k+1)}}\cdots p_{\alpha_n\alpha_1\cdots \alpha_{(k-1)}}}.$$

\begin{defn}
	For any pair of cyclic orders $\alpha = (\alpha_1\cdots\alpha_n)$ and $\beta = (\beta_1\cdots \beta_n)$ on $\lbrack n\rbrack = \{1,\ldots, n\}$, define a function $m^{(k)}(\alpha,\beta)$, by 
	\begin{eqnarray*}
		m^{(k)}(\alpha,\beta) & = & \sum_{\text{soln}}\left(\frac{1}{\det'\left(\Phi^{(k)}\right)}PT^{(k)}\left(\alpha\right)PT^{(k)}\left(\beta\right)\bigg\vert_{\text{soln}}\right),
	\end{eqnarray*}
	Here the sum is over all solutions to the scattering equations. 

	For short, when $\alpha = \beta = (1,2,\ldots, n)$, we define 
	$$m^{(k)}_n = m^{(k)}(\alpha,\alpha).$$
\end{defn}
Sometimes to be explicit we include the dependence on the kinematic data, that is $m^{(k)}_n(\mathfrak{s})$.

A second definition using the positive tropical Grassmannian was introduced in \cite{CE2020B}.  Let $P_J(y_{i,j})$ be the tropicalization of the Pl\"{u}cker variable $p_J(x_{i,j})$, evaluated on the parametrization $\mathbf{M}$ in Appendix \ref{sec:positive parametrization}.

Following \cite{CE2020B}, denote by $\mathcal{F}^{(k)}_n$ the function,
$$\mathcal{F}^{(k)}_n(y;\mathfrak{s}) = \sum_{J \in \binom{\lbrack n\rbrack}{k}}P_J(y)\mathfrak{s}_J,$$
where $P_J(y)$ are the (tropicalized) Pl\"{u}cker coordinates, evaluated on the above parametrization.  An equivalent expression for this is
$$\mathcal{F}^{(k)}_n(y; \mathfrak{s}) = \sum_{J\in \binom{\lbrack n\rbrack}{k}^{nf}} \omega_J(y) \eta_J(\mathfrak{s}),$$
where $\omega_J$ are certain tropical planar cross-ratios, see \cite{Early2020WeightedBladeArrangements,E2021} for details.  When there is no risk of confusion, we use the same notation for the parametrization of the positive tropical Grassmannian, which takes values in $\mathbb{R}^{\binom{n}{k}}$ rather than in the dual (kinematic) space,
$$\mathcal{F}^{(k)}_n(y) = \sum_{J\in \binom{\lbrack n\rbrack}{k}^{nf}} \omega_J(y) \mathfrak{h}_J,$$

\begin{prop}[\cite{CE2020B}]
	We have 
	$$m^{(k)}_n(\mathfrak{s}) = \int_{\mathbb{T}^{k-1,n-k}}\exp\left(-\mathcal{F}^{(k)}_n(y; \mathfrak{s})\right)dy.$$
	for all values of $(\mathfrak{s}_J)$ such that the integral converges.  Here we remind that $\mathbb{T}^{k-1,n-k} = \mathbb{T}^{n-k} \times \cdots \times\mathbb{T}^{n-k}$ is the Cartesian product of $k-1$ tropical tori $\mathbb{T}^{n-k} = \mathbb{R}^{n-k}\slash \mathbb{R}(1,\ldots, 1)$.
\end{prop}

\end{document}